\documentclass[11pt]{article}

\usepackage{amsmath}
\usepackage{amssymb}
\usepackage{amsthm}
\usepackage{amsfonts}
\usepackage{amscd}
\usepackage{thmtools} 
\usepackage[numbers,sort&compress]{natbib}
\usepackage{hyperref}
\usepackage{cleveref} 
\usepackage[dvipsnames]{xcolor} 
\usepackage{geometry}
\usepackage{enumitem}
\usepackage[T1]{fontenc} 
\usepackage{tikz}
\usepackage[medium]{titlesec} 
\usepackage[font=it, labelfont={bf,it}]{caption} 
\usepackage[nottoc]{tocbibind} 
\usepackage{tocloft} 
\usepackage{appendix}

\declaretheorem[style=plain,name=Theorem]{theorem}
\declaretheorem[style=plain,name=Corollary]{corollary}
\declaretheorem[style=plain,name=Theorem]{theoremext}

\declaretheorem[style=plain,name=Lemma]{lemma}
\declaretheorem[style=plain,name=Proposition,sibling=lemma]{proposition}

\declaretheorem[style=definition,name=Remark,sibling=lemma]{remark}
\declaretheorem[style=definition,name=Example,sibling=lemma]{example}
\Crefname{theorem}{Theorem}{Theorems}
\Crefname{theoremext}{Theorem}{Theorems}
\Crefname{corollary}{Corollary}{Corollaries}
\Crefname{lemma}{Lemma}{Lemmas}
\Crefname{proposition}{Proposition}{Propositions}
\Crefname{definition}{Definition}{Definitions}
\Crefname{remark}{Remark}{Remarks}
\Crefname{example}{Example}{Examples}

\hypersetup{colorlinks,linkcolor=RoyalBlue,citecolor=PineGreen,urlcolor=RoyalBlue}
\newcommand{\myspace}{\setlength{\abovedisplayskip}{1mm}\setlength{\belowdisplayskip}{0mm}}
\titlespacing*{\section}{0pt}{0mm}{0mm}
\titlespacing*{\subsection}{0pt}{0mm}{0mm}
\usetikzlibrary{positioning} 
\usetikzlibrary{arrows.meta} 

\newenvironment{Mlist}{\begin{itemize}[topsep=0pt,itemsep=4pt,leftmargin=7mm]}{\end{itemize}}
\newenvironment{claims}{\begin{enumerate}[topsep=0pt,itemsep=0pt,leftmargin=8mm,label=\textnormal{\textbf{(\alph*)}},ref=(\alph*)]}{\end{enumerate}}

\newcommand{\fig}[3]{\includegraphics[height=#1cm, width=#2cm]{img/#3}}

\newcommand{\RL}[2]{\Cref{lem:#1}\ref{lem:#1:#2}}
\newcommand{\RLS}[4]{Lemmas~\ref{lem:#1}\ref{lem:#1:#2} and~\ref{lem:#3}\ref{lem:#3:#4}}

\newcommand{\C}{\mathbb{C}}
\newcommand{\R}{\mathbb{R}}

\renewcommand{\H}{\mathbb{H}}
\renewcommand{\P}{\mathbb{P}}
\renewcommand{\S}{\mathbb{S}}
\newcommand{\Z}{\mathbb{Z}}

\newcommand{\bC}{\mathbf{C}}
\newcommand{\bP}{\mathbf{P}}
\newcommand{\bR}{\mathbf{R}}
\newcommand{\bS}{\mathbf{S}}

\newcommand{\cN}{\mathcal{N}}
\newcommand{\cE}{\mathcal{E}}
\newcommand{\cL}{\mathcal{L}}
\newcommand{\cR}{\mathcal{R}}
\newcommand{\cU}{\mathcal{U}}
\newcommand{\cH}{\mathcal{H}}

\newcommand{\tA}{\tilde{A}}
\newcommand{\tB}{\tilde{B}}

\newcommand{\tQ}{\tilde{Q}}
\newcommand{\tS}{\tilde{S}}
\newcommand{\tT}{\tilde{T}}

\newcommand{\oA}{{\overline{A}}}
\newcommand{\oB}{{\overline{B}}}
\newcommand{\oC}{{\overline{C}}}
\newcommand{\oD}{{\overline{D}}}
\newcommand{\oL}{{\overline{L}}}
\newcommand{\oR}{{\overline{R}}}
\newcommand{\oh}{{\overline{h}}}
\newcommand{\op}{{\overline{p}}}
\newcommand{\oq}{{\overline{q}}}
\newcommand{\oalpha}{{\overline{\alpha}}}
\newcommand{\obeta}{{\overline{\beta}}}
\newcommand{\oa}{{\overline{a}}}
\newcommand{\ob}{{\overline{b}}}

\newcommand{\vc}{\vec{c}}
\newcommand{\vv}{\vec{v}}
\newcommand{\vw}{\vec{w}}

\newcommand{\HC}{\mathbb{H}_{\C}}
\newcommand{\qi}{\mathbf{i}}
\newcommand{\qj}{\mathbf{j}}
\newcommand{\qk}{\mathbf{k}}
\newcommand{\ii}{\mathfrak{i}}

\renewcommand{\c}{\colon}
\newcommand{\dto}{\dasharrow}
\newcommand{\aut}{\operatorname{Aut}}

\newcommand{\set}[2]{\left\{#1 ~|~ #2\right\}}
\newcommand{\Sing}{\operatorname{Sing}}
\newcommand{\df}{\emph}
\newcommand{\sgn}{\operatorname{sgn}}

\newcommand{\Wlog}{without loss of generality }
\newcommand{\resp}{respectively}
\newcommand{\st}{such that }

\renewcommand{\k}{k}

\renewcommand{\l}{\ell}
\newcommand{\p}{\varepsilon}
\newcommand{\bas}[1]{\langle #1\rangle}

\newcommand{\END}{\hfill $\vartriangleleft$}

\begin{document}
\myspace

\begin{center}
~\vspace{-10mm}\\
\LARGE
Bivariate quaternionic factorizations and surfaces that decompose into two circles
\\[5mm]\large
Johanna Frischauf, Niels Lubbes, Hans-Peter Schröcker
\\[2mm]\large
\today
\end{center}
~\vspace{-10mm}
\begin{abstract}
We present an algebraic and geometric condition
for bivariate quaternionic polynomials of arbitrary bidegree
to have a univariate linear left or right factor.
We apply this quaternionic factorization theorem to the bidegree (1,1) case
and recover a classical theorem of Clifford in elliptic geometry.
By applying to the bidegree (2,2) case,
we obtain decompositions into two circles
of celestial surfaces, namely surfaces in the 3-dimensional sphere
that contain two circles through a general point.
This results in an alternative proof and refinement for
a theorem by Skopenkov and Krasauskas from 2019,
which states that a non-quartic celestial surface is
Möbius equivalent to either the pointwise product of circles in the unit-quaternions,
or an inverse stereographic projection
of the pointwise sum of circles in Euclidean space.
Our proposed method extends this decomposition result to the quartic case
and we show that surfaces are, up to Möbius equivalence and stereographic projections,
not both a sum and product of circles.
\\[2mm]
{\bf Keywords:}
quaternions,
factorization of bivariate quaternionic polynomials,
\\Möbius geometry,
real rational surfaces, celestial surfaces
\\[2mm]
{\bf MSC2010:}
12D05, 
51B10, 
51M15, 
14J26 
\end{abstract}

\vspace{-3mm}
\begingroup
\def\addvspace#1{\vspace{-1.4mm}}
\tableofcontents
\endgroup
\vspace{2mm}

\section{Introduction}

In this article, we address the following two informal problems, which
will be made precise later in this introduction:
\begin{Mlist}
\item[1.]
\textit{Determine a necessary and sufficient algebraic condition for
bivariate quaternionic polynomials to admit
a univariate linear left or right factor.}
\item[2.]
\textit{Characterize surfaces in $\R^3$ that \df{decompose into two circles}, namely surfaces
that are M\"obius equivalent to the Zariski closure of either
the stereographic projection of a pointwise product of circles in the unit-quaternions~$S^3$,
or the pointwise sum of circles in~$\R^3$ (see \Cref{fig:decompose}).
}
\end{Mlist}
\begin{figure}[!ht]
\centering
\setlength{\tabcolsep}{1cm}
\begin{tabular}{cc}
\fig{2.8}{4}{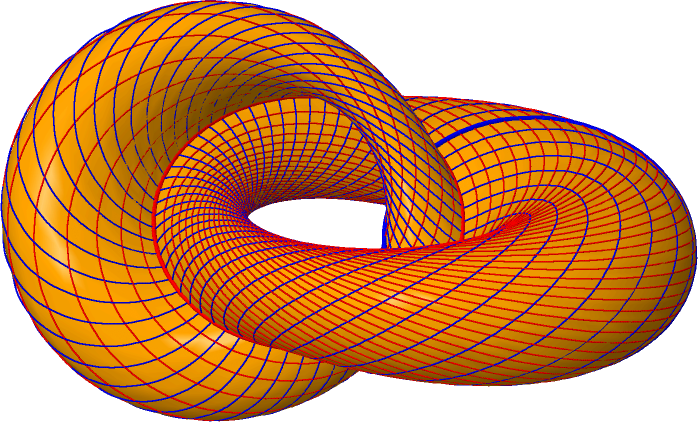}&
\fig{2.8}{4}{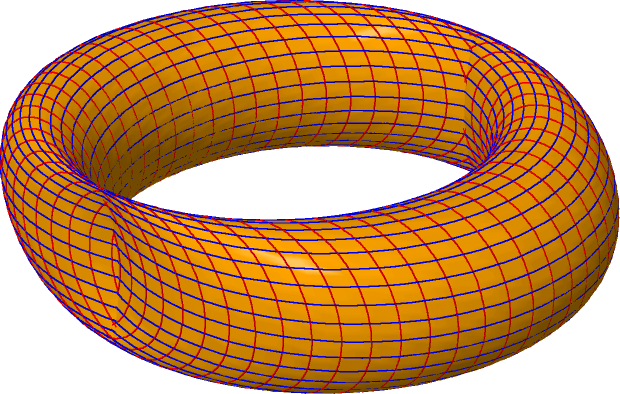}
\\
product of circles & sum of circles
\end{tabular}
\\
\caption{Examples of surfaces that decompose into two circles.}
\label{fig:decompose}
\end{figure}
The first problem is non-trivial, since multivariate quaternionic polynomials
are not commutative and do not form a unique factorization domain.
It was solved for the bidegree~$(*,1)$ case in 2019
by Skopenkov and Krasauskas \citep[Lemma~2.9]{2018sko}
and as an application they partially solved the second problem
(see \citep[Main~Theorem~1.1]{2018sko}):
{\it A surface in $\R^3$ that is covered by two pencils of circles either
\begin{Mlist}
\item decomposes into two circles, or
\item is the stereographic projection of a quartic surface in~$S^3$.
\end{Mlist}}
We remark that a surface that decomposes into two circles is covered by at least
two pencils of circles. Surfaces are assumed to be real, algebraic and Zariski closed, but not necessarily smooth.

In this article, we solve the first problem for bivariate quaternionic polynomials of arbitrary bidegree (see \Cref{thm:fct}).
A geometric interpretation of this factorization result is as follows:
if a quaternionic polynomial defines a regular parametrization of a surface in projective
3-space, then this surface intersected with some fixed quadratic surface of signature~$(4,0)$
contains a pair of complex conjugate lines for each univariate linear left/right factor
(see \Cref{cor:fct}).
When applied to quaternionic polynomials of
bidegree~$(1,1)$, we recover a classical result of Clifford in elliptic geometry
(see \Cref{cor:clifford}).

We solve the second problem by applying \Cref{cor:fct} to bidegree~$(2,2)$ polynomials
(see \Cref{thm:c} and \Cref{cor:c}):
{\it A surface in~$\R^3$ that is covered by two pencils of circles belongs up to M\"obius equivalence
to exactly one of the following types of surfaces in~$\R^3$:
\begin{Mlist}
\item a surface that decomposes into two circles,
\item a smooth cubic surface that is covered by exactly 2 or 6 pencils of circles, or
\item a quadratic surface that is neither a sphere nor a one-sheeted circular hyperboloid.
\end{Mlist}}
In addition to the first item, we show that a surface does not decompose both as a sum and as a product of two circles.
Notice that the last two items classify the non-decomposable surfaces among those covered
by two pencils of circles.
Since decomposable surfaces are necessarily covered by two pencils of circles,
we indeed solved the second problem.
\Cref{thm:c} recovers and refines \citep[Main~Theorem~1.1]{2018sko} with an alternative proof
and
confirms \citep[Conjecture~1]{2023trans} which was stated by a coauthor of the current article
(see \Cref{rmk:known} for more details).

Before we discuss our main results in \textsection\ref{sec:fct}, \textsection\ref{sec:clifford}
and \textsection\ref{sec:c} in more detail, we first present an overview for the state of the art.

\subsection{State of the art}
\label{sec:art}
Univariate quaternionic polynomials admit a non-unique factorization into linear factors \citep[Theorem~5.2]{2008}
and thus the first problem is solved for bidegrees $(*,0)$.
The factorizations of univariate dual quaternionic polynomials have been studied
in \cite{2021slice,2019} and are used in \cite{2018kempe} to construct mechanical linkages
that draw given rational space curves.

For bivariate quaternionic polynomials of bidegree~$(*,1)$,
a necessary and sufficient condition for the existence
of linear univariate factors was provided by~\citep[Lemma~2.9]{2018sko} (see also \citep[Corollary~3.2]{2022n1}).
Bivariate quaternionic polynomials whose norm polynomial
factors into univariate factors are studied in \cite{2022tech}.
\Cref{thm:fct} does not make assumptions on the norm
polynomial and applies to quaternionic polynomials of arbitrary bidegree.

Surfaces that are covered by two pencils of circles are called \df{celestial surfaces}
and are algebraic by \citep[Theorem~1]{2001}.
A non-quadratic celestial surface in the 3-dimensional unit-sphere
is of degree either 4 or~8
and contains at most~6 circles through a general point \cite{2021circ,1987,1980}.
Those of degree 4 are called \df{Darboux cyclides}
and have been investigated by
mathematicians such as
Dupin~(1822),
Villarceau~(1848),
Kummer~(1865),
Maxwell~(1868),
Casey~(1870),
Cayley~(1873),
Clifford~(1873) and
Darboux~(1880)
\citep[Introduction]{2019mor}.
Artists sculptured around 1580
approximations of non-trivial circles on a Darboux cyclide in a staircase near the Strasbourg Cathedral
(see \cite{2014rt}, \citep[page~101]{2010ber}),
and
\citep[Main~Theorem~1.1]{2018sko} shows
more than 400 years later
that a celestial surface is
either a Darboux cyclide or decomposes into two circles.
\Cref{thm:c} extends this result by also characterizing the decomposable Darboux cyclides.
Celestial surfaces underwent a recent revival in geometric modeling and discrete differential geometry
\cite{1995,2012ort,2014rin,2014bez,2019mor}.
Of particular interest are their applications in
kinematics~\cite{2018kin,2019boh,2020kin} and architecture~\cite{2011arc,2012web}.


From an algebro geometric point of view,
we consider rational maps of component degree~$(2,2)$ into a 3-dimensional projective hyperquadric.
\Cref{thm:fct} applies as such parametric maps translate into bivariate quaternionic polynomials.
A related result is the birational classification at \citep[Theorem~3]{2018kol} of
rational maps of component degree~$2$
from the projective plane into a 4-dimensional projective hyperquadric.
In case this hyperquadric is of signature~$(4,1)$, then \citep[Corollary~5.1]{2018sko}
characterizes such maps in terms of products of quaternionic polynomials.
We remark that \Cref{lem:coord} of our alternative proof applies to rational maps of component degree~$(2,2)$
into hyperquadrics of arbitrary dimension.

We now proceed by stating our main result in \textsection\ref{sec:fct} (\Cref{thm:fct}),
its geometric interpretations in \textsection\ref{sec:clifford} (Corollaries~\ref{cor:fct} and \ref{cor:clifford}),
and its application to celestial surfaces in \textsection\ref{sec:c} by specializing
to the bidegree~$(2,2)$ case (\Cref{thm:c}).
We conclude the introduction with an overview of the remaining article at \textsection\ref{sec:overview}.

\subsection{The factorization theorem}
\label{sec:fct}
The algebra of \df{quaternions} consists of a
vector space $\H:=\bas{1,\qi,\qj,\qk}_\R$ and a multiplication $\_\cdot\_\c \H\times \H\to \H$
that is defined by $\qi^2=\qj^2=\qk^2=\qi\cdot\qj\cdot\qk=-1$ with $1\in \H$ being the multiplicative unit.
The algebra of \df{complex quaternions} $\HC$ is defined as
the set of pairs $(a,b)\in\H^2$ with componentwise addition and the multiplication defined by
$(a,b)\cdot(p,q):=(a\cdot p-b\cdot q, a\cdot q + b\cdot p)$.
We identify $(0,1)\in\HC$ with the imaginary unit $\ii\in\C$
so that $\HC=\set{a+\ii\,b}{a,b\in\H}$. Notice that $\ii\neq\qi$.

The \df{conjugate} of a quaternion or complex quaternion~$h=h_1+h_2\,\qi+h_3\,\qj+h_4\,\qk$
is defined as $h^*:=h_1-h_2\,\qi-h_3\,\qj-h_4\,\qk$.
We remark that $h\cdot h^*=h_1^2+h_2^2+h_3^2+h_4^2$.
The \df{complex conjugate} of a complex quaternion $h=h_1+h_2\,\qi+h_3\,\qj+h_4\,\qk$
is defined as $\overline{h}:=\oh_1+\oh_2\,\qi+\oh_3\,\qj+\oh_4\,\qk$,
where the bar denotes conjugation of complex numbers.

We denote by $\HC[s,t]$ the bivariate polynomials with coefficients in $\HC$
and we assume that the coefficients commute with the variables~$s$ and~$t$.
The \df{conjugate}~$Q^*$ of the polynomial~$Q\in \HC[s,t]$ is defined as the polynomial in~$\HC[s,t]$
that is obtained by conjugating the quaternionic coefficients of~$Q$.

We write $P\mid Q$ for $P\in\C[s,t]$ and $Q\in \HC[s,t]$ if $Q=P\cdot\tQ$ for some $\tQ\in\HC[s,t]$.
It is not necessary to distinguish whether $P$ is a left or a right factor of~$Q$, as
$P\cdot\tQ=\tQ\cdot P$ by \RL{basic}{a}.
We call a polynomial $Q\in\HC[s,t]$ \df{primitive} if $P\nmid Q$ for all non-constant
polynomials $P\in\C[s,t]$. All definitions for polynomials in~$\HC[s,t]$ also apply
to polynomials in $\H[s,t]\subset \HC[s,t]$.

\begin{theorem}
\label{thm:fct}
Let $Q\in \H[s,t]$ be primitive. We write
$Q=Q_0+Q_1\cdot s+\cdots+Q_d\cdot s^d$
with $d\in\Z_{>0}$ and $Q_i\in\H[t]$ for all $i\in I:=\{0,\ldots,d\}$.
Let $P\in\R[t]$ be a monic quadratic polynomial.
\begin{claims}
\item\label{thm:fct:a}
$P\mid Q_i\cdot Q_j^*$ for all $i,j\in I$
if and only if
$Q=\tQ\cdot (t-h)$
for some $\tQ\in\H[s,t]$ and $h\in\H$
\st $P=(t-h)\cdot (t-h)^*$.

\item\label{thm:fct:b}
$P\mid Q_i^*\cdot Q_j$ for all $i,j\in I$
if and only if
$Q=(t-h)\cdot \tQ$
for some $\tQ\in\H[s,t]$ and $h\in\H$
\st $P=(t-h)\cdot (t-h)^*$.


\end{claims}
\end{theorem}

\begin{remark}
\label{rmk:fct}
The following assertion follows from the proof of \Cref{thm:fct}
and can be used to weaken the assumptions for the $\Rightarrow$ direction of \Cref{thm:fct}\ref{thm:fct:a}:
\textit{If $P\nmid Q_i$ for some $i\in I$ and $P\mid Q_i\cdot Q_j^*$ for all~$j\in I$,
then $P\mid Q_i\cdot Q_j^*$ for all~$i,j\in I$.}
A similar assertion can be used to weaken the $\Rightarrow$ direction of \Cref{thm:fct}\ref{thm:fct:b}:
\textit{If $P\nmid Q_i$ for some $i\in I$ and $P\mid Q_i^*\cdot Q_j$ for all~$j\in I$,
then $P\mid Q_i^*\cdot Q_j$ for all~$i,j\in I$.}
\END
\end{remark}

We illustrate \Cref{thm:fct} with examples of bidegrees~$(1,1)$ and~$(2,2)$.

\begin{example}
\label{exm:fct1}
Let $Q=Q_1\cdot s+Q_0$ in $\H[s,t]$ be of bidegree~$(1,1)$,
where
\[
Q_1:=(\qi-\qk)\cdot t - (1+\qi-\qj-\qk)
\quad\text{and}\quad
Q_0:=-(1+\qi+\qj-\qk)\cdot t+2\,(1+\qi).
\]
Let $P:=t^2-2\,t+2$. We find that $Q_0\cdot Q_0^*=2\cdot Q_1\cdot Q_1^*=4\cdot P$,
\[
Q_0\cdot Q_1^*=-2\cdot P\cdot (\qk+1)
\quad\text{and}\quad
Q_1\cdot Q_0^*=2\cdot P\cdot (\qk-1).
\]
\Cref{thm:fct}\ref{thm:fct:a} predicts that $Q$ has a right factor $t-h$ for some $h\in\H$
\st $P=(t-h)\cdot(t-h)^*$.
Indeed,
\[
Q=(s-1-\qk)\cdot(\qi-\qk)\cdot(t-1-\qk)
\]
and $h:=1+\qk$.
\END
\end{example}

\begin{example}
\label{exm:fct2}
Let $Q:=Q_2\cdot s^2+Q_1\cdot s+Q_0$ in $\H[s,t]$ be of bidegree (2,2), where
\begin{align*}
Q_2:=&\ 2\,\qk\cdot t^2 + 2\,(\qj - 2\,\qk - 1)\cdot t + 2\, (1 - \qi - \qj + \qk),\\
Q_1:=&\ 2\,(\qj - 2\,\qk - 1)\cdot t^2 + 4\,(2-2\,\qj +\qk)\cdot t+4\,(\qi + \qj - 2), \\
Q_0:=&\ 2\,(1+\qi - \qj + \qk)\cdot t^2 +4\,(\qj - \qi - 2)\cdot t + 8.
\end{align*}
Let $P:=t^2-2\,t+2$. We find that $P\nmid Q_2$,
\[
Q_2\cdot Q_0^*=4\cdot P^2\cdot(1 -\qi - \qj + \qk),\quad
Q_2\cdot Q_1^*=4\cdot P^2\cdot(\qi - \qk - 2),\quad
Q_2\cdot Q_2^*=4\cdot P^2,
\]
and thus $P\mid Q_i\cdot Q_j^*$ for all $i,j\in\{0,1,2\}$ by \Cref{rmk:fct}.
\Cref{thm:fct}\ref{thm:fct:a} predicts that $Q$ has a right factor $t-h$ for some $h\in\H$
\st $P=(t-h)\cdot(t-h)^*$.
In this example, $h:=1-\qi$ and
\[
Q=2\,\qk\cdot(s-1+\qi)\cdot(s-1+\qk)\cdot(t-1+\qk)\cdot(t-1+\qi),
\]
and thus $Q$ even admits a factorization with linear factors.
\END
\end{example}

\subsection{Geometric corollaries}
\label{sec:clifford}

If $S\subset\C^4\cong\HC$ is a subset,
then we denote by $\P(S)$ the Zariski closure of its projectivization in
the \df{complex projective space}~$\P^3$.
If $Q\in\HC[s,t]$ is non-zero, then we define
\[
[Q]:=\P(\set{Q(s,t)\in\HC}{s,t\in\C}).
\]
Notice that if $q,r\in \HC$ correspond to non-zero vectors in~$\C^4$ \st their span
is a line through the origin,
then $[q]=[r]$.
The \df{null quadric} in $\P^3$ is defined as
\[
\cN:=\P(\set{h\in\HC}{h\cdot h^*=0})=\set{x\in\P^3}{x_0^2+x_1^2+x_2^2+x_3^2=0}.
\]
If $\alpha\in\HC$ such that $\alpha\cdot\alpha^*=0$, then
\[
\cL_\alpha:=\P(\set{h\in\HC}{h\cdot\alpha=0})
\qquad\text{and}\qquad
\cR_\alpha:=\P(\set{h\in\HC}{\alpha\cdot h=0}),
\]
are called a \df{left generator} and a \df{right generator}, respectively.
They are the two complex lines in~$\cN$ that pass through the complex point~$[\alpha^*]\in\cN$
(see \RL{N}{a}).

The following corollary is the geometric version of \Cref{thm:fct}.

\begin{corollary}
\label{cor:fct}
Let $Q\in\H[s,t]$ be primitive and $p\in\C$.
\begin{claims}
\item\label{cor:fct:a}
$[Q(s,p)]\subseteq\cL_\alpha$
and
$[Q(s,\op)]\subseteq\cL_\oalpha$
for some $[\alpha^*]\in\cN$
if and only if
$Q=\tQ\cdot (t-h)$ for some $\tQ\in\H[s,t]$ and $h\in\H$
\st $(p-h)\cdot (p-h)^*=0$.

\item\label{cor:fct:b}
$[Q(s,p)]\subseteq\cR_\alpha$
and
$[Q(s,\op)]\subseteq\cR_\oalpha$
for some $[\alpha^*]\in\cN$
if and only if
$Q=(t-h)\cdot\tQ$ for some $\tQ\in\H[s,t]$ and $h\in\H$,
where $(p-h)\cdot (p-h)^*=0$.

\end{claims}
\end{corollary}

\begin{remark}
The assumption for the $\Rightarrow$ direction of \Cref{cor:fct}\ref{cor:fct:a}
is equivalent to:
\textit{$[Q(s,p)]$ is contained in a left generator for some $p \in \C$.}
If $[Q(s,p)]$ is not a complex point, then the right factor $t-h$
corresponds to two complex conjugate left generators in $[Q]\cap\cN$.
The situation for \Cref{cor:fct}\ref{cor:fct:b} is analogous.
\END
\end{remark}

\begin{corollary}
\label{cor:clifford}
If $Q\in\H[s,t]$ is primitive and of bidegree~$(1,1)$,
then the following are equivalent:
\begin{enumerate}[topsep=0pt,itemsep=2pt,label=(\roman*),ref=(\roman*)]
\item\label{i}
$[Q(s,p)], [Q(q,t)]\subset\cN$ are contained in complex lines for some $p,q\in\C\setminus\R$.
\item\label{ii}
$Q\in\{S\cdot T,T\cdot S\}$ for some linear polynomials~$S\in\H[s]$ and~$T\in\H[t]$.
\item\label{iii}
$Q\cdot Q^*=A\cdot B$ for some quadratic polynomials~$A\in\R[s]$ and $B\in\R[t]$.
\end{enumerate}
\end{corollary}

\begin{remark}
\Cref{cor:clifford} specializes \Cref{cor:fct} to bidegree~$(1,1)$
and recovers both the \citep[Splitting Lemma~2.6]{2018sko}
and
the classical theorem \citep[7.94]{1998} in elliptic geometry by William Kingdon Clifford (1845--1879~CE)
(see also \citep[p.241]{1968}).
\END
\end{remark}

\begin{example}
\label{exm:fct3}
Let us consider the bidegree~$(1,1)$ polynomial from \Cref{exm:fct1}:
\[
Q:=(s-1-\qk)\cdot(\qi-\qk)\cdot(t-1-\qk).
\]
Since $t-h$ with $h:=1+\qk$ is a right factor,
\Cref{cor:fct}\ref{cor:fct:a} predicts that
$[Q(s,p)]\subseteq\cL_\alpha$ for
some $p\in\C$ and $[\alpha^*]\in\cN$.
Moreover, $p$ is a root of
\[
(t-h)\cdot(t-h)^*=t^2 - 2\,t + 2=(t-1-\ii)\cdot(t-1+\ii)
\]
and thus $p=1+\ii$ without loss of generality.
Since $[Q(p,p)]=(-\ii:0:0:1)$ is contained in $\cL_\alpha$,
we may assume \Wlog that $[\alpha^*]=[Q(p,p)]$, which
implies that $\alpha=-\ii-\qk$.
\Cref{cor:fct}\ref{cor:fct:a} confirms that
\[
[Q(s,p)]=\set{(-s+\op:\ii\,s+\op:s-p:-\ii\,s+p)}{s\in\C}\subseteq\cL_\alpha.
\]
Similarly, \Cref{cor:fct}\ref{cor:fct:b} with $s$ and $t$ interchanged confirms that
\[
[Q(p,t)]=\set{(-t+\op:\ii\,t+\op:-t+p:-\ii\,t+p)}{t\in\C}\subseteq\cR_\alpha.
\]
In fact, $[Q(s,p)]=\cL_\alpha$, $[Q(p,t)]=\cR_\alpha$
so that $[Q(s,\op)]=\cL_\oalpha$ and $[Q(\op,t)]=\cR_\oalpha$.
In accordance with \Cref{cor:clifford}, we have
\[
[Q]\cap\cN=\cL_\alpha\cup\cL_\oalpha\cup\cR_\alpha\cup\cR_\oalpha
\]
and
$
Q\cdot Q^*=(s^2 - 2\,s + 2)\cdot (t^2 - 2\,t + 2).
$
\END
\end{example}

\begin{example}
\label{exm:fct4}
Let us consider the bidegree~$(2,2)$ polynomial from \Cref{exm:fct2}:
\[
Q:=2\,\qk\cdot(s+\qi-1)\cdot(s-1+\qk)\cdot(t-1+\qk)\cdot(t-1+\qi).
\]
We set $p:=1+\ii$ so that it is a root of $(t-h)\cdot (t-h)^*$, where $h:=1-\qi$.
We may assume \Wlog that $[\alpha^*]$ equals $[Q(p,p)]=(0:0:-\ii:1)$
and thus we set $\alpha:=(\ii)\,\qj-\qk$.
As predicted by \Cref{cor:fct}\ref{cor:fct:a}, we observe that the right factor~$t-h$ of $Q$ corresponds
to
\[
[Q(s,p)]=\set{
(
-\ii\,s-\op:
-s+p:
\ii\,s-p:
-s+\op
)
}{s\in\C}
\subseteq \cL_\alpha.
\]
In fact, we find that
$[Q(s,p)]=\cL_\alpha$
so that $[Q(s,\op)]=\cL_\oalpha$.
\END
\end{example}

\subsection{Application to celestial surfaces}
\label{sec:c}
An \df{inversion} with respect to a sphere~$S\subset \R^3$
with center~$c$ and radius~$r$ is the
map
\[
f\c \R^3\setminus\{c\}\to\R^3\setminus\{c\}
\]
such that $||x-c||\cdot ||f(x)-c||=r^2$ and the vectors $x-c$
and $f(x)-c$ are codirected for all $x\in \R^3\setminus\{c\}$.
Such a map exchanges the interior and exterior of~$S$
and sends a circle or line to a circle or line.
We call two surfaces in $\R^3$ \df{M\"obius equivalent} if
one surface is mapped to the other by a composition of inversions.

Let $\mu\c S^3\dto\R^3$ with $\mu(y):=(y_1,y_2,y_3)/(1-y_4)$
denote the \df{stereographic projection}
from the point $(0,0,0,1)$ on the \df{3-dimensional unit-sphere}~$S^3\subset\R^4$.

A surface~$Z\subset\R^3$ is called \df{$\lambda$-circled}
if the Zariski closure of $\mu^{-1}(Z)$
contains at least $\lambda\in \Z_{\geq 0}\cup\{\infty\}$ circles through a general point.
If $\lambda\in\Z_{\geq0}$, then we assume that $Z$ is not $(\lambda+1)$-circled.
If $\lambda\geq 2$, then we call $Z$ \df{celestial}.

A \df{Darboux cyclide} $Z\subset \R^3$
is a $\lambda$-circled surface \st
$\lambda\geq 1$ and $\deg\mu^{-1}(Z)=4$.
Smooth 4-circled, 5-circled and 6-circled Darboux cyclides
are called \df{ring cyclides}, \df{Perseus cyclides} and \df{Blum cyclides}, \resp.
A \df{CH1 cyclide} is a 3-circled Darboux cyclide that is M\"obius equivalent
to a one-sheeted circular hyperboloid.
See \Cref{fig:darboux} for examples and \cite{1980} for the origin of the names.
\begin{figure}[!ht]
\setlength{\tabcolsep}{3mm}
\begin{tabular}{cccc}
\fig{3}{3}{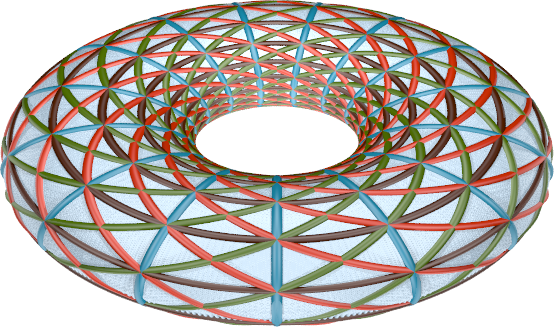} &
\fig{3}{3}{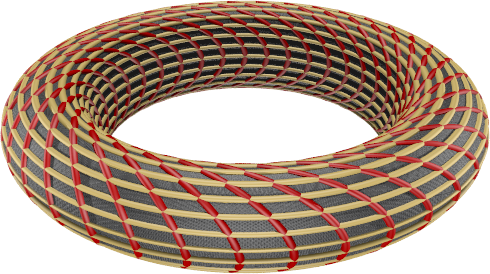}&
\fig{3}{3}{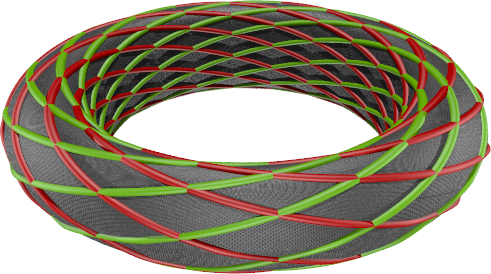}&
\fig{3}{3}{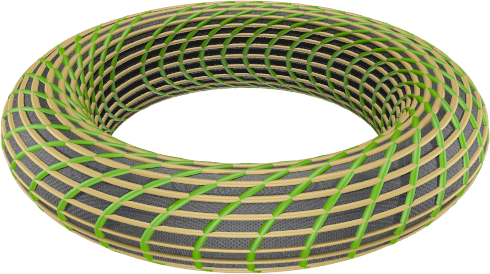}
\\
ring cyclide & & Perseus cyclide &
\\
\fig{3}{3}{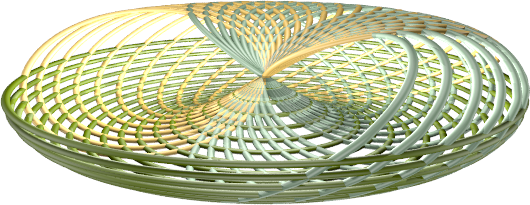} &
\fig{3}{3}{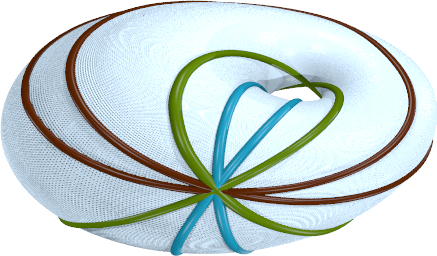} &
\fig{3}{3}{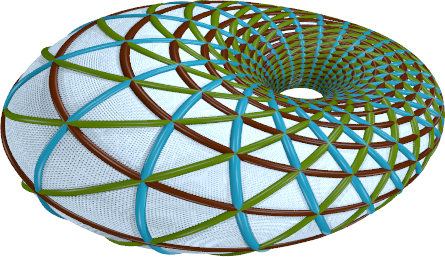} &
\fig{3}{3}{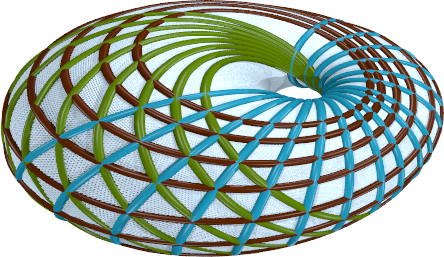}
\\
CH1 cyclide & & Blum cyclide &
\end{tabular}
\caption{Circles on four different types of Darboux cyclides.}
\label{fig:darboux}
\end{figure}

\begin{figure}[!hb]
\setlength{\tabcolsep}{1mm}
\begin{tabular}{ccccc}
\fig{2.8}{2.8}{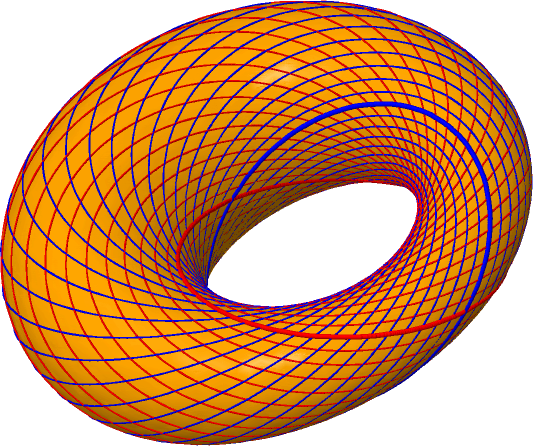} &
\fig{2.8}{2.8}{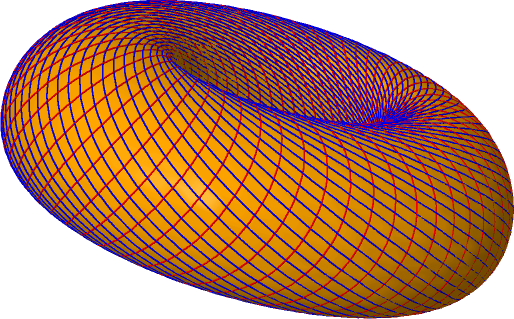} &
\fig{2.8}{2.8}{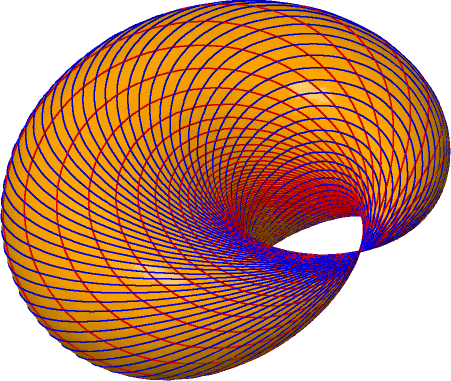} &
\fig{2.8}{2.8}{prod-8} &
\fig{2.8}{2.8}{sum-8}
\\
ring cyclide & Perseus cyclide & CH1 cyclide & $\deg(C\cdot D)=8$ & $A+B$
\end{tabular}
\caption{The first four surfaces are stereographic projections of $C\cdot D$
for some circles $C,D\subset S^3$.
The rightmost surface decomposes as $A+B$ for some circles $A,B\subset\R^3$.
An illustrated circle
corresponds to either
$\{c\}\cdot D$,
$C\cdot\{d\}$,
$\{a\}+B$, or
$A+\{b\}$,
for some
$a\in A$, $b\in B$, $c\in C$ and $d\in D$.
}
\label{fig:ABCD}
\end{figure}

Since $\R^4\cong\H$ as vector spaces, we may identify the unit-sphere $S^3\subset\R^4$ with
the \df{unit quaternions}~$\set{h\in\H}{h\cdot h^*=1}$.
If $A,B\subset \R^3$ and $C,D\subset S^3$, then
we denote by $A+B$ and $C\cdot D$
the Zariski closures of
\[
\set{a+b \in \R^3}{ a\in A \text{ and } b\in B}
\quad\text{and}\quad
\set{c\cdot d \in S^3}{ c\in C \text{ and } d\in D},
\]
\resp.
See \Cref{fig:ABCD} for examples.

The following result is an application of \Cref{cor:fct} to the bidegree (2,2) case.

\begin{theorem}
\label{thm:c}
Suppose that $Z\subset \R^3$ is a $\lambda$-circled surface such that $\lambda\geq 2$
and let $d$ denote the degree of the inverse stereographic projection~$\mu^{-1}(Z)$.
\begin{claims}
\item\label{thm:c:a}
We have $d\in \{2,4,8\}$ and
if $d$ equals 4 or 8, then $\lambda\leq 6$ and $\lambda=2$, \resp.
\item\label{thm:c:b}
If $d\neq 4$, then $Z$ is M\"obius equivalent to either $A+B$ or $\mu(C\cdot D)$
for some circles $A,B\subset \R^3$ and $C,D\subset S^3$.
\item\label{thm:c:c}
If $d\neq 8$,
then $Z$ is M\"obius equivalent to $\mu(C\cdot D)$
for some circles~$C,D\subset S^3$
if and only if
$Z$ is either a Perseus cyclide, ring cyclide or CH1 cyclide.
\item\label{thm:c:d}
If $Z=A+B$ for some circles $A,B\subset\R^3$,
then $d\neq 4$ and there do not exist circles $C,D\subset S^3$ \st
$Z$ is M\"obius equivalent to $\mu(C\cdot D)$.
\end{claims}
\end{theorem}

\begin{remark}
\label{rmk:known}
\Cref{thm:c}\ref{thm:c:a} also follows from \citep[Theorem~1]{2021circ}
and answers a question in \citep[\textsection5]{2012web}, but we included a short proof.
\Cref{thm:c}\ref{thm:c:b} is essentially equivalent to \citep[Main~Theorem~1.1]{2018sko},
but we present an alternative proof using \Cref{cor:fct}.
\Cref{thm:c}\ref{thm:c:c} confirms \citep[Conjecture~1]{2023trans},
which did not surrender using only intersection theory.
It follows from \citep[Theorem~1]{2001} that analytic surfaces that are covered
by two analytic families of circles are algebraic
and therefore we restrict ourselves to the algebraic setting
(see also \citep[\textsection4]{2018sko}).
The proof for Theorems~\ref{thm:c}\ref{thm:c:a} and~\ref{thm:c}\ref{thm:c:c} depend
on \cite{2001} and \cite{2023trans}, respectively (see Theorems~\ref{thm:A} and~\ref{thm:B}).
\END
\end{remark}

\begin{corollary}
\label{cor:c}
Suppose that $Z\subset\R^3$ is a $\lambda$-circled surface.
Then $\lambda\geq 2$ if and only if
$Z$ is M\"obius equivalent to exactly one of the following types of surfaces:
\begin{enumerate}[topsep=0pt,itemsep=2pt,label=(\roman*),ref=(\roman*)]
\item\label{c:i} $\mu(C\cdot D)$ for some circles $C,D\subset S^3$. In this case $\lambda\in\{2,3,4,5\}$.
\item\label{c:ii} $A+B$ for some circles $A,B\subset\R^3$. In this case $\lambda\in\{2,\infty\}$.
\item\label{c:iii} The cubic surface
$
\set{(x,y,z)\in\R^3}{z(x^2+y^2+z^2)+a\,x^2+b\,y^2+c\,z=0}
$
for some $a,b,c\in \R^3$ \st
$abc(a-b)(a^2+c)(b^2+c)\neq 0$.
This surface is smooth and $\lambda\in\{2,6\}$.

\item\label{c:iv} The quadratic surface
$
\set{(x,y,z)\in\R^3}{a\,x^2+b\,y^2+c\,z^2+d\,z+e=0}
$
for some $a\in\R_{>0}$, $b\in\R_{>0}\cup\{-1\}$ and $c,d,e\in\{-1,0,1\}$
\st either
\\
$(a-b)(b+1)c(d^2-1)e(c+e-2)\neq 0$ or
$(b+1)(c-1)(d^2-1)(e-1)(c^2e^2-1)\neq 0$ or
$(a-b)(c^2-1)(d^2+d)(e^2-1)\neq 0$.
In this case $\lambda\in\{2,3,4\}$.
\end{enumerate}
\end{corollary}

Notice that \Cref{cor:c} implies that a 5-circled surface in $S^3$ is
M\"obius equivalent to the pointwise quaternionic product of two circles.

\begin{example}
\label{exm:CH1}
Let $S\cdot\tS\cdot\tT\cdot T\in\H[s,t]$ be the polynomial from \Cref{exm:fct2}, where
\[
S:=\qk\cdot\sqrt{2}\cdot(s-1+\qi),\qquad
\tS:=\sqrt{2}\cdot(s-1+\qk),\qquad
\tT:=t-1+\qk,\qquad
T:=t-1+\qi.
\]
In order to show that this polynomial is associated to a celestial surface, we set
\[
U:=\frac{S\cdot\tS}{S\cdot S^*}
\qquad\text{and}\qquad
V:=\frac{\tT\cdot T}{T\cdot T^*}.
\]
Properties of the quaternionic conjugation (see \RL{basic}{a}) show that
\[
(S\cdot\tS)\cdot (S\cdot\tS)^*=(S\cdot S^*)\cdot(\tS\cdot\tS^*)
\quad\text{and}\quad
(\tT\cdot T)\cdot (\tT\cdot T)^*=(\tT\cdot\tT^*)\cdot( T\cdot T^*).
\]
It follows that $U\cdot U^*=V\cdot V^*=1$, since
\[
S\cdot S^*=\tS\cdot\tS^*=2\cdot(s^2-2\,s+2)
\qquad\text{and}\qquad
\tT\cdot\tT^*= T\cdot T^*=t^2-2\,t+2.
\]
Hence,
$U,V\c\R\to S^3$ define rational functions into the
unit-quaternions
\[
S^3\cong\set{h\in\H}{h\cdot h^*=1}.
\]
We consider the following surfaces, where
$.^-$ denotes the Zariski closure and
$c:=(0,0,0,1)$ is the center of the stereographic projection~$\mu$:
\[
Y:=\set{U(s)\cdot V(t)}{s,t\in\R}^-\subset S^3
\qquad\text{and}\qquad
Z:=\mu(Y\setminus\{c\})^-\subset\R^3.
\]
The numerators and denominators of~$U$ and~$V$ are polynomials of degree two,
which implies that $U$ and~$V$ parametrize irreducible conics with real points.
Therefore, $U$ and $V$ parametrize circles in~$S^3$ so that $Z$ is a $\lambda$-circled celestial surface.
\Cref{thm:c} predicts that $\deg Y\in\{2,4,8\}$ and $\lambda\in\{2,3,4,5\}$.
In particular, notice that if $\lambda\neq 2$, then $\deg Y\neq 8$ by \Cref{thm:c}\ref{thm:c:a}
and thus $\lambda\neq 6$ by \Cref{thm:c}\ref{thm:c:c}.
For this example, we find that
\[
Z=\set{(x,y,z)\in\R^3}{x^2+y^2-z^2-1=0},
\]
and thus $Z$ is a CH1 cyclide so that $\deg Y=4$ and $\lambda=3$.
The CH1 cyclide in \Cref{fig:ABCD} is M\"obius equivalent to~$Z$
and realized as the stereographic projection of $Y$ from a center in $S^3\setminus Y$.
\END
\end{example}

For the remaining examples in \Cref{fig:ABCD},
we refer to \citep[Example~26]{2023trans}.
We now conclude this introduction with an overview of the remaining article.

\subsection{Overview}
\label{sec:overview}

In \textsection\ref{sec:factor}, we prove \Cref{thm:fct}
and its geometric Corollaries~\ref{cor:fct} and~\ref{cor:clifford}.
The remaining sections are dedicated to proving \Cref{thm:c}.
In \textsection\ref{sec:model}, we consider a model
of M\"obius geometry
in terms of a hyperquadric in $\P^4$ of signature~$(1,4)$.
Now suppose that $X$ is the projective model of a celestial surface
embedded into this hyperquadric.
First, we assume that there exists a birational morphism $\varphi\c \P^1\times\P^1\to X$
whose components are of bidegree~$(2,2)$.
We show in \textsection\ref{sec:morphism} that
under this assumption there exists a hyperplane section of $X$
that consists of four complex double lines that are parameter lines with respect to~$\varphi$.
These complex double lines lie in a quadratic surface.
If this quadratic surface is smooth, then we show in \textsection\ref{sec:E}
that $X$ is M\"obius equivalent to a product of circles.
If this quadratic surface is singular, then we show in \textsection\ref{sec:U}
that $X$ is M\"obius equivalent to a sum of circles.
In \textsection\ref{sec:proof}, we conclude the proof of \Cref{thm:c} and \Cref{cor:c}.
In particular, we know from \Cref{thm:A}
that $\deg X\in\{2,4,8\}$ and if $\deg X=8$,
then $\varphi$ indeed exists.
The classification in case $\deg X\in\{2,4\}$
is completed using \Cref{thm:B} and concerns birational
maps $\P^1\times\P^1\dto X$ that are not everywhere defined.
\Cref{thm:A,thm:B} depend on external references.

The proof of \Cref{thm:B} is delayed until
Appendix~\ref{app:B}.
The proof of \Cref{thm:c}\ref{thm:c:d}, namely that $X$ cannot be
M\"obius equivalent to both a sum and product of circles,
is delayed until Appendix~\ref{app:ABCD}.

We did an effort to make the article self-contained and
accessible to a wide audience.
For the reader's convenience, we sometimes include proofs of known results
together with a reference to the original result.

\section{Bivariate quaternionic factorizations}
\label{sec:factor}

The goal of this section is to prove \Cref{thm:fct}
and its geometric Corollaries~\ref{cor:fct} and~\ref{cor:clifford}.

Before we start with a basic algebraic lemma for the non-expert,
notice that if $R\in\HC[s,t]$, then
$R(s,t)=R_0+R_1\cdot s +\cdots+R_d\cdot s^d$ with $R_i\in\HC[t]$ for all $i\in\{1,\ldots,d\}$.
Indeed, the variables $s$ and $t$ commute with the coefficients by assumption.
The \df{evaluation}~$R(c,t)$ is for all $c\in \C$ is defined as $R_0+R_1\cdot c +\cdots+R_d\cdot c^d$.
Notice that $Q\cdot c=c\cdot Q$ for all $c\in\C$ and $Q\in\HC[t]$ by \RL{basic}{a} below.
In this article, we do not evaluate at quaternions in~$\HC\setminus\C$.

\begin{lemma}
\label{lem:basic}
\hspace{1em}
\begin{claims}
\item\label{lem:basic:a}
For all $Q,H\in\HC[s,t]$ and $P\in\C[s,t]$,
$P\cdot Q=Q\cdot P$,
$P^*=P$,
$Q\cdot Q^*=Q^*\cdot Q\in \C[s,t]$,
$(H\cdot Q)^*=Q^*\cdot H^*$
and $(H+Q)^*=H^*+Q^*$.
If $Q\in \H[s,t]$, then $Q\cdot Q^*\in \R[s,t]$.

\item\label{lem:basic:b}
For all $Q,R\in\HC[t]$ and $p\in\C$
we have $Q(p)\cdot R(p)=(Q\cdot R)(p)$.
Moreover, $Q(p)=0$ if and only if $t-p \mid Q$.
If $Q\in\H[t]$ and $p\in\C\setminus\R$,
then $Q(p)=0$ if and only if $(t-p)\cdot(t-\overline{p})\mid Q$.

\item\label{lem:basic:c}
If $Q\in\HC[t]$ is linear, then $[Q]$ is a complex point if and only if
$t-p\mid Q$ for some $p\in\C$.

\item\label{lem:basic:d}
If $A\in\H[t]$ and $B\in\R[t]$, then there exist
$Q,R\in\H[t]$ \st $A=Q\cdot B+R$ and $\deg R<\deg B$.

\item\label{lem:basic:e}
For all $A,B\in\H[t]$,
we have $\deg(A\cdot B)=\deg A+\deg B$ and $\deg A^*=\deg A$.


\end{claims}
\end{lemma}

\begin{proof}
\ref{lem:basic:a}
The identities are direct consequences of the definition
of the quaternion multiplication and conjugation, and our convention that
the complex quaternionic coefficients commute with the variables $s$ and $t$.

\ref{lem:basic:b}
The identity $Q(p)\cdot R(p)=(Q\cdot R)(p)$ follows from the fact that
complex constants commute with complex quaternions.
Let us write $Q=Q_1+Q_2\,\qi+Q_3\,\qj+Q_4\,\qk$
for some complex polynomials~$Q_1,\ldots,Q_4\in\C[t]$.
We observe that $Q(p)=0$
if and only if
$Q_1(p)=\cdots=Q_4(p)=0$
if and only if
$Q_i=\tQ_i\cdot(t-p)$ for some $\tQ_i\in\C[t]$
if and only if
$t-p\mid Q$.
Moreover, if $Q \in \H[t]$ and $p \in \C \setminus \R$,
then $Q_1(\op)=\cdots=Q_4(\op)=0$ so that $(t-p)\cdot(t-\op)\mid Q$.

\ref{lem:basic:c}
By assumption
$Q=Q_1+Q_2\,\qi+Q_3\,\qj+Q_4\,\qk$
and
$Q_i=a_i\cdot (t-b_i)$ for some $a_i,b_i\in\C$ \st $a_j\neq 0$
for at least one index $j\in\{1,2,3,4\}$.
We observe that
$[Q]$ is a complex point
if and only if
$[Q]=(a_1\cdot (t-p):\ldots:a_4\cdot (t-p))=(a_1:\ldots:a_4)$ for some $p\in \C$
if and only if
$t-p\mid Q$ for some $p\in \C$.

\ref{lem:basic:d}
Let us write $A=A_1+A_2\,\qi+A_3\,\qj+A_4\,\qk$ with $A_1,\ldots,A_4\in\R[t]$.
Let $Q_i$ and $R_i$ be the polynomial quotient and remainder, \resp,
of $A_i$ and $B$
so that $A_i=Q_i\cdot B+R_i$ and $\deg R_i<\deg B$
for all $1\leq i\leq 4$
by Euclid's division lemma for polynomials in~$\R[t]$.
Thus if
$Q=Q_1+Q_2\,\qi+Q_3\,\qj+Q_4\,\qk$ and
$R=R_1+R_2\,\qi+R_3\,\qj+R_4\,\qk$,
then $A=Q\cdot B+R$ and $\deg R<\deg B$ as asserted.

\ref{lem:basic:e}
Notice that there exist~$a,b\in\H\setminus\{0\}$
\st the leading term of $A\cdot B$ is $a\cdot b\cdot t^{\deg A+\deg B}$
and the leading term of $A^*$ is $a^*\cdot t^{\deg A}$.
Since $a$ and $b$ are non-zero real quaternions,
we find that $a\cdot b\neq 0$ and thus $\deg(A\cdot B)=\deg A+\deg B$ and $\deg A^*=\deg A$.
\end{proof}

Since $A\cdot A^*\in\R[t]$ for all $A\in\H[t]$ by \RL{basic}{a},
we observe that $A\cdot A^*$ factors into at most quadratic irreducible factors
by the fundamental theorem of algebra for real polynomials.
The following lemma is essentially the known fundamental theorem of algebra for quaternionic
polynomials (see \cite{1966} and \citep[Theorem~5.2]{2008}).

\begin{lemma}
\label{lem:uni}
Suppose that $A\in\H[t]$ is a primitive monic polynomial of degree $d\geq 1$.
\begin{claims}
\item\label{lem:uni:a}
For all tuples $(B_1,\ldots,B_r)$ of irreducible polynomials in~$\R[t]$
\st
\[
A\cdot A^*=B_1\cdot\ldots\cdot B_r,
\]
we have $r=d$ and
there exist~$a_1,\ldots,a_r\in\H$
\st $B_i=(t-a_i)\cdot (t-a_i)^*$ for all $1\leq i\leq r$ and
$
A=(t-a_r)\cdots(t-a_1).
$

\item\label{lem:uni:b}
If there exist $E,F\in\H[t]$ and $e,f\in\H$ \st
\[
A=E\cdot(t-e)=F\cdot(t-f)
\quad\text{and}\quad
(t-e)\cdot(t-e)^*=(t-f)\cdot(t-f)^*,
\]
then $e=f$.

\end{claims}
\end{lemma}


\begin{proof}
\ref{lem:uni:a}
If $d=1$, then the assertion is a direct consequence of the assumptions.
By induction hypothesis, we suppose that Assertion~\ref{lem:uni:a} is true for~$d-1$ and~$d\geq 2$.
Notice that $1\leq \deg B_i\leq 2$ by the fundamental theorem of algebra for real polynomials.
By \RL{basic}{d}, there exist $Q,R\in\H[t]$ \st
\[
A=Q\cdot B_1+R
\quad\text{and}\quad
\deg R<\deg B_1\leq 2.
\]
By assumption, $A$ is primitive and thus $A\neq Q\cdot B_1$, which implies that $R\neq 0$.
We apply \RL{basic}{a}, and deduce that $R\cdot R^*$ is equal to
\begin{gather*}
(A-Q\cdot B_1)\cdot (A-Q\cdot B_1)^*
=
A\cdot A^*-B_1\cdot( A\cdot Q^*+ Q\cdot A^* - B_1\cdot Q\cdot Q^*).
\end{gather*}
Hence, $R\cdot R^*=B_1\cdot C$, where $C:=B_2\cdots B_r - A\cdot Q^*- Q\cdot A^* + B_1\cdot Q\cdot Q^*$.
Since $\deg(R\cdot R^*)=2\cdot \deg R=\deg B_1+\deg C$ by \RL{basic}{e},
it follows that $(\deg R,~\deg B_1,~\deg C)=(1,2,0)$.
Hence, $R=c\cdot(t-a_1)$ for some~$c,a_1\in\H$ with $c\neq 0$.
By \RL{basic}{a}, we have
$R\cdot R^*=(t-a_1)\cdot (t-a_1)^*\cdot c\cdot c^*$
and
$(t-a_1)\cdot (t-a_1)^*=(t-a_1)^*\cdot (t-a_1)$.
This implies that $B_1=(t-a_1)^*\cdot (t-a_1)$ and $C=c\cdot c^*$ so that
\begin{equation}
A=Q\cdot B_1+R=Q\cdot (t-a_1)^*\cdot (t-a_1)+c\cdot(t-a_1).
\end{equation}
Therefore, there exists a monic polynomial $\tA\in\H[t]$ of degree $d-1$ without real factors \st
$A=\tA\cdot (t-a_1)$.
We apply the induction hypothesis to~$\tA$
and conclude Assertion~\ref{lem:basic:a}.
In particular, we observe that $\deg B_i=2$ for all $1\leq i\leq r$ so that $r=d$ by \RL{basic}{e}.

\ref{lem:uni:b}
Let $B:=(t-e)\cdot(t-e)^*=(t-f)\cdot(t-f)^*$.
We apply \RL{basic}{d} and obtain $A=Q\cdot B+R$ with $\deg R<\deg B=2$.
Since $t-e$ is a right factor of~$A$ by assumption and of~$B$ by \RL{basic}{a},
we find that it is also a right factor of~$R$.
Similarly, we find that $t-f$ is a right factor of~$R$.
Therefore, $R=u\cdot(t-e)=v\cdot(t-f)$.
We compare coefficients and find that $u=v$ and $-u\cdot e=-v\cdot f$.
We multiply both sides with $-u^*/(u\cdot u^*)=-u^{-1}$ and conclude that $e=f$.
\end{proof}

The following result is classical and can be traced back to
Clifford (see \citep[\textsection7.9 and 7.93]{1998}).
We adapted the algebraic the proof at \citep[Proposition~1]{2020}.

\begin{lemma}
\label{lem:N}
Suppose that $\alpha,h,q\in\HC\setminus\{0\}$ such that $\alpha\cdot\alpha^*=0$.
\begin{claims}
\item\label{lem:N:a}
The generators $\cL_\alpha$ and $\cR_\alpha$ are the two complex lines in the hyperquadric $\cN$
that pass through the complex point $[\alpha^*]\in\cN$.

\item\label{lem:N:b}
If $[h],[q]\in \cL_\alpha$, then $h\cdot q^*=0$.\\
If $[h],[q]\in \cR_\alpha$, then $q^*\cdot h=0$.

\item\label{lem:N:c}
If $h\cdot q^*=h\cdot h^*=q\cdot q^*=0$, then $[h],[q]\in \cL_{q^*}$.\\
If $q^*\cdot h=h\cdot h^*=q\cdot q^*=0$, then $[h],[q]\in \cR_{q^*}$.

\item\label{lem:N:d}
If $[h]\in\cL_\alpha$ and $q\cdot h\neq 0$, then $[q\cdot h]\in \cL_\alpha$.\\
If $[h]\in\cR_\alpha$ and $h\cdot q\neq 0$, then $[h\cdot q]\in \cR_\alpha$.
\end{claims}
\end{lemma}

\begin{proof}
\ref{lem:N:a}
The map $\HC\to\HC$ that sends $h$ to $h\cdot\alpha$ is linear
with respect to the underlying vector space~$\langle 1,\qi,\qj,\qk \rangle_\C$
and has kernel
\[
L_\alpha:=\set{h\in\HC}{h\cdot\alpha=0}.
\]
Let
$\beta:=\alpha^*$ and
$V_\beta:=\langle \beta,~\qi\cdot\beta,~\qj\cdot\beta,~\qk\cdot\beta\rangle_\C$.
By assumption,
$
\alpha\cdot\alpha^*=
\beta\cdot\beta^*=
\beta\cdot\alpha=0$ and thus $V_\beta\subseteq L_\alpha$.

Now suppose by contradiction that $\dim V_\beta<2$.
In this case,
\[
\beta=c_1\cdot\qi\cdot\beta=c_2\cdot\qj\cdot\beta=c_3\cdot\qk\cdot\beta
\]
for some $c_1,c_2,c_3\in\C$. Since
\[
\qi=\qj\cdot\qk=-\qk\cdot\qj,\quad
\qj=\qk\cdot\qi=-\qi\cdot\qk,\quad
\qk=\qi\cdot\qj=-\qj\cdot\qi,\quad
\qi^2=\qj^2=\qk^2=-1,
\]
we find that
\[
\begin{array}{rlrlrlrlr}
                \beta &=&          \beta_1 &+&         \beta_2\cdot\qi &+&         \beta_3\cdot\qj &+&         \beta_4\cdot\qk, \\
c_1\cdot\qi\cdot\beta &=& -c_1\cdot\beta_2 &+& c_1\cdot\beta_1\cdot\qi &-& c_1\cdot\beta_4\cdot\qj &+& c_1\cdot\beta_3\cdot\qk, \\
c_2\cdot\qj\cdot\beta &=& -c_2\cdot\beta_3 &+& c_2\cdot\beta_4\cdot\qi &+& c_2\cdot\beta_1\cdot\qj &-& c_2\cdot\beta_2\cdot\qk, \\
c_3\cdot\qk\cdot\beta &=& -c_3\cdot\beta_4 &-& c_3\cdot\beta_3\cdot\qi &+& c_3\cdot\beta_2\cdot\qj &+& c_3\cdot\beta_1\cdot\qk.
\end{array}
\]
By comparing the coefficients, we see that
$\beta_1=-c_i\cdot\beta_{i+1}$,
$\beta_{i+1}=c_i\cdot\beta_1$
and thus
$\beta_1\neq 0$ and
$c_i^2=-1$
for all $1\leq i\leq 3$.
This implies that $\beta_2=\beta_3=\beta_4=\pm\ii\cdot\beta_1$
and thus
\[
\beta\cdot\beta^*=\beta_1^2+\beta_2^2+\beta_3^2+\beta_4^2=-2\cdot\beta_1^2=0.
\]
We arrived at a contradiction as $\beta_1\neq 0$.

We established that $\dim L_\alpha\geq \dim V_\beta\geq 2$ and thus $\dim \cL_\alpha\geq 1$
as $\cL_\alpha=\P(L_\alpha)$ by definition.
Suppose that $h\in L_\alpha$ so that $h\cdot\alpha=0$ and thus $h^*\cdot h\cdot\alpha=h\cdot h^*\cdot\alpha=0$.
Since $h\cdot h^*\in\C$ and $\alpha\neq 0$, we find that $h\cdot h^*=h^*\cdot h=0$.
By definition, $\cN=\P(\set{h\in\HC}{h\cdot h^*=0})$
and thus $\cL_\alpha\subset\cN$.
Since $\cN$ is a smooth hyperquadric of~$\P^3$, it does not contain complex planes.
It follows that $\cL_\alpha$ is a complex line through the complex point~$[\alpha^*]$.
Similarly, we show that $\cR_\alpha\subset\cN$ is a complex line through the complex point~$[\alpha^*]$.
Since $\cL_\alpha\neq\cR_\alpha$, we concluded the proof for Assertion~\ref{lem:N:a}.

For Assertion~\ref{lem:N:b},
notice that if $[h],[q]\in \cL_\alpha$, then
$[h]\in \cL_{q^*}$ by Assertion~\ref{lem:N:a} and thus $h\cdot q^*=0$ by definition.
The remaining Assertions \ref{lem:N:b}, \ref{lem:N:c} and \ref{lem:N:d}
are now a straightforward consequence of the definitions.
\end{proof}

The following lemma is \citep[Lemma~2.2]{2022abc}, but we included a proof for
the reader's convenience.

\begin{lemma}[abc-Lemma]
\label{lem:abc}
If $a,b,c\in\HC$ \st $a\cdot b\cdot c=0$ and $a\cdot a^*=b\cdot b^*=c\cdot c^*=0$,
then either $a\cdot b=0$ or $b\cdot c=0$.
\end{lemma}

\begin{proof}
Suppose that $a\cdot b\neq 0$, $b\neq 0$ and $c\neq0$, otherwise there is nothing to prove.
By definition, $\cR_b\subseteq \cR_{a\cdot b}$ and $[c]\in \cR_{a\cdot b}$.
It follows from \RL{N}{a} that $\cR_b=\cR_{a\cdot b}$ and thus $[c]\in\cR_b$
so that $b\cdot c=0$.
\end{proof}

The following lemma follows from \citep[Proposition~2.1]{2016new}.
See also \citep[Lemma~2]{2024} for an insightful proof that relies on
a non-commutative version of the Euclidean algorithm.
For self-containment, we provide another alternative proof using the abc-Lemma.

\begin{lemma}[AB-Lemma]
\label{lem:AB}
Suppose that $A,B\in\H[t]$ and let $P\in\R[t]$ be a monic quadratic irreducible polynomial.
If $P\mid A\cdot B$, $P\nmid A$ and $P\nmid B$,
then there exist $\tA,\tB\in\H[t]$ and~$a\in\H$ such that
\[
A=\tA\cdot (t-a),\quad
B=(t-a)^*\cdot\tB \quad\text{and}\quad
P=(t-a)\cdot (t-a)^*.
\]
\end{lemma}

\begin{proof}
Let $p\in \C$ \st $P(p)=0$.
Since $P\mid A\cdot B$, we have $A(p)\cdot B(p)=0$ by \RL{basic}{b}.
Hence, $A^*(p)\cdot A(p)\cdot B(p)=0$.
Suppose by contradiction that $A^*(p)\cdot A(p)\neq 0$.
By \RL{basic}{a}, we have $A^*(p)\cdot A(p)\in \C$ and thus $B(p)=0$.
We arrive at a contradiction as $P\mid B$ by \RL{basic}{b}.
We established that $A^*(p)\cdot A(p)=0$ and thus
$(A\cdot A^*)(p)=0$ by \RL{basic}{b} so that $P\mid A\cdot A^*$.
Similarly, we find that $P\mid B\cdot B^*$.
It now follows from \Cref{lem:uni} that
there exist~$\tA, \tB \in \H[t]$ and $a, b \in \H$ \st
\[
A=\tA\cdot (t-a),\qquad
B=(t-b)\cdot \tB,\qquad
P=(t-a)\cdot (t-a)^*=(t-b)\cdot (t-b)^*.
\]
Notice that if $(t-b)=(t-a)^*$, then we concluded the proof.

First suppose that $(p-a)\cdot(p-b)=0$
so that $P\mid (t-a)\cdot(t-b)$ by \RL{basic}{b}.
Since $P=(t-a)\cdot (t-a)^*$ and thus monic and quadratic, we find that
\[
P=(t-a)\cdot(t-b)=(t-a)\cdot (t-a)^*.
\]
This implies that $(t-b)=(t-a)^*$ and thus we concluded the proof for this case.

Now suppose by contradiction that $(p-a)\cdot(p-b)\neq 0$.
Since $A=\tA\cdot (t-a)$ and $B=(t-b)\cdot \tB$, we find that
\[
A(p)\cdot B(p)=\tA(p)\cdot(p-a)\cdot(p-b)\cdot \tB(p)=0.
\]
This implies that
\[
\tA(p)^*\cdot \tA(p)\cdot(p-a)\cdot(p-b)\cdot \tB(p)\cdot\tB(p)^*=0
\]
and thus either $\tA(p)^*\cdot\tA(p)=0$, $\tB(p)\cdot\tB(p)^*=0$ or both.
We make a case distinction.

If $\tA(p)^*\cdot\tA(p)=0$ and $\tB(p)\cdot\tB(p)^*\neq 0$,
then $\tA(p)\cdot(p-a)\cdot(p-b)=0$
and thus $A(p)=\tA(p)\cdot (p-a)=0$ by the abc-\Cref{lem:abc}.
This implies that $P\mid A$ by \RL{basic}{b} and thus we arrived at a contradiction.

If $\tB(p)^*\cdot\tB(p)=0$ and $\tA(p)\cdot\tA(p)^*\neq 0$,
then $(p-a)\cdot(p-b)\cdot\tB(p)=0$ using a similar argument as before.
Hence, $P\mid B$ and we arrive again at a contradiction.

Finally, suppose that $\tA(p)^*\cdot\tA(p)=0$ and $\tB(p)\cdot\tB(p)^*=0$.
We apply the abc-\Cref{lem:abc} to $\tA(p)\cdot(p-a)\cdot(p-b)\cdot \tB(p)=0$
and find that either $\tA(p)\cdot(p-a)\cdot(p-b)=0$ or
$(p-a)\cdot(p-b)\cdot\tB(p)=0$. As before, these cases imply that $P\mid A$ or $P\mid B$
and thus we arrived at a contradiction.

As all cases are considered, we concluded the proof.
\end{proof}

\begin{proof}[Proof of \Cref{thm:fct} and \Cref{rmk:fct}]
\ref{thm:fct:a}
Let us first prove the $\Rightarrow$ direction.

As an initial step, we suppose by contradiction that $P$ is reducible.
In this case, $P=(t-\alpha)\cdot (t-\beta)$ for some~$\alpha,\beta\in\R$
so that $(t-\alpha)\mid Q_i\cdot Q_i^*$ for all~$i\in I$.
We deduce that $Q_i$ is not primitive, since otherwise $Q_i\cdot Q_i^*$
has only quadratic factors by \RL{uni}{a}.
Therefore, $(t-\alpha)\mid Q_i$ for all $i\in I$.
We arrived at a contradiction as $Q$ is primitive
and thus $P$ must be irreducible.

As $Q$ is primitive, there exists $i\in I$ \st $P\nmid Q_i$.
Since $P$ is irreducible and $P\mid Q_i\cdot Q_i^*$ by assumption,
it follows from \RL{uni}{a} that
there exist $h\in \H$ and $\tQ_i\in\H[t]$
\st $Q_i=\tQ_i\cdot (t-h)$ and $P=(t-h)\cdot (t-h)^*$.
Let $j\in I\setminus\{i\}$. We make a case distinction on
$P\mid Q_j$ and $P\nmid Q_j$.

If $P\mid Q_j$,
then $Q_j=\tQ_j\cdot (t-h)$ for some $\tQ_j\in\H[t]$ by \RL{basic}{a}.

Now suppose that $P\nmid Q_j$.
In this case, $P\nmid Q_j^*$ and $P\mid Q_i\cdot Q_j^*$.
We deduce from the AB-\Cref{lem:AB} and \RL{basic}{a}
that $Q_i=Q_i'\cdot(t-a)$ and $Q_j=\tQ_j\cdot (t-a)$
for some $a\in\H$ and $Q_i',\tQ_j\in\H[t]$ \st $P=(t-a)\cdot(t-a)^*$.
Recall that $Q_i=\tQ_i\cdot (t-h)$ and $P=(t-h)\cdot(t-h)^*$.
It follows from \RL{uni}{b} that $a=h$ and thus $Q_j=\tQ_j\cdot (t-h)$.

We established that $Q_k=\tQ_k\cdot (t-h)$ for all $k\in I$
and thus $Q=\tQ\cdot (t-h)$ for some~$\tQ\in\H[s,t]$.

Next, we proceed with the $\Leftarrow$ direction.
In this case, for all $i\in I$ there exists a polynomial $\tQ_i\in\H[t]$ \st
$Q_i=\tQ_i\cdot (t-h)$.
Hence, $Q_i\cdot Q_j^*=P\cdot\tQ_i\cdot\tQ_j^*$ for all $i,j\in I$
by \RL{basic}{a} so that $P\mid Q_i\cdot Q_j^*$.
This concludes the proof of Assertion~\ref{thm:fct:a}.

The proof of Assertion~\ref{thm:fct:b} is analoguous to the proof of Assertion~\ref{thm:fct:a}.
\Cref{rmk:fct} is a direct consequence of the proofs of the
$\Rightarrow$ directions.
\end{proof}

\begin{lemma}
\label{lem:plane}
Suppose that $F\c \C^2\to \C^4$
is a form of degree $d\in\Z_{>0}$
such that
\[
F(x,y)=\vc_0\cdot y^d+ \vc_1\cdot x\cdot y^{d-1} +\cdots+ \vc_d\cdot x^d,
\]
with coefficients $\vc_i\in\C^4$ for all $i\in I:=\{0,\ldots,d\}$
and let $V\subset \C^4$ be some subspace of dimension at most two.
\begin{claims}
\item\label{lem:plane:a}
If $F(\C^2)=V$, then $\vc_i\in V$ for all $i\in I$.

\item\label{lem:plane:b}
If $\vc_i\in V$ for all $i\in I$, then $F(\C^2)\subseteq V$.
\end{claims}
\end{lemma}

\begin{proof}
Since $F(\C^2)=\langle\vv,\vw\rangle_\C$ for some $\vv,\vw\in\C^4$ we deduce that
\[
F(x,y)=f(x,y)\cdot\vv + g(x,y)\cdot\vw
\]
for some smooth functions $f, g\c\C^2\to\C$.
We take the $i$-th partial derivative of $F$ with respect to the variable~$x$
and evaluate at $(0,1)$ so that
\[
i!\cdot\vc_i=
(\partial_x^i F)(0,1)=
(\partial_x^if)(0,1)\cdot\vv + (\partial_x^ig)(0,1)\cdot\vw\in F(\C^2),
\]
for all $i\in I$.
This concludes the proof for Assertion~\ref{lem:plane:a}.
Assertion~\ref{lem:plane:b} is an immediate consequence of the definitions.
\end{proof}

\begin{proof}[Proof of \Cref{cor:fct}]
Let us write
\[
Q(s,t)=Q_0+Q_1\cdot s+\cdots+Q_d\cdot s^d
\]
with $d\in\Z_{\geq 0}$ and $Q_i\in\H[t]\subset\HC[t]$ for all $i\in I:=\{0,\ldots,d\}$.
Because $\HC$ is as a vector space equal to $\langle 1, \qi, \qj, \qk \rangle_\C$,
we may identify $Q_i(p)\in \HC$ with a vector in~$\C^4$ for all~$i\in I$.
Let $F(x,y):=y^d\cdot Q(x/y,p)$ be the homogenization of $Q(s,p)$.

\ref{cor:fct:a}
Let us first prove the $\Rightarrow$ direction.
By assumption, $[Q(s,p)]\subseteq\cL_\alpha\subset\cN$
and thus $\set{F(x,y)}{x,y\in \C}\subset\C^4$
is a subspace of dimension at most two.
We apply \RL{plane}{a} to $F(x,y)$
and deduce that either $Q_i(p)=0$ or $[Q_i(p)]\in \cL_\alpha$,
for all~$i\in I$.
Hence, $Q_i(p)\cdot Q_j(p)^*=0$ for all $i,j\in I$ by \RL{N}{b}.
It follows from \RL{basic}{b} that
$P\mid Q_i\cdot Q_j^*$, where $P:=(t-p)\cdot(t-\op)$.
The $\Rightarrow$ direction for Assertion~\ref{cor:fct:a} now follows from \Cref{thm:fct}\ref{thm:fct:a}.

Next, we proceed with the $\Leftarrow$ direction.
Notice that $P:=(t-h)\cdot(t-h)^*\in\R[t]$ is a monic quadratic polynomial by \RLS{basic}{a}{basic}{e}.
We know from \Cref{thm:fct}\ref{thm:fct:a} that $P\mid Q_i\cdot Q_j^*$ for all~$i,j\in I$.
Hence, $(Q_i\cdot Q_j^*)(p)=Q_i(p)\cdot Q_j(p)^*=0$ for all~$i,j\in I$ by \RL{basic}{b}.
\RL{N}{c} implies that either $Q_i(p)=0$ or
$[Q_i(p)]\in \cL_\alpha$ for all $i\in I$ and some $[\alpha^*]\in\cN$.
As $\cL_\alpha$ does not contain real points,
it follows that $[Q_i(\op)]\in \cL_\oalpha$ for all $i\in I$.
We apply \RL{plane}{b} to $F(x,y)$ and find that
$[Q(s,p)]\subseteq\cL_\alpha$
and
$[Q(s,\op)]\subseteq\cL_\oalpha$, which concludes the proof of Assertion~\ref{cor:fct:a}.

The proof of Assertion~\ref{cor:fct:b} is analoguous to the proof of Assertion~\ref{cor:fct:a}.
\end{proof}

\begin{proof}[Proof of \Cref{cor:clifford}]
\ref{i}$\Rightarrow$\ref{ii}:
It follows from \RL{N}{a} that either $[Q(s,p)]\subseteq\cL_\alpha$
or $[Q(s,p)]\subseteq\cR_\alpha$ for some $\alpha\in\HC$ \st $\alpha\cdot\alpha^*=0$.
If $[Q(s,p)]\subseteq\cL_\alpha$, then $[Q(s,\op)]\subseteq\cL_\oalpha$
as $[Q(s,p)]$ is non-real.
It follows from \Cref{cor:fct}\ref{cor:fct:a} that $Q=S\cdot T$ for some linear polynomials $S\in\H[s]$ and $T\in\H[t]$.
Similarly, if $[Q(s,p)]\subseteq\cR_\alpha$,
then $Q=T\cdot S$ by \Cref{cor:fct}\ref{cor:fct:b} with $s$ and $t$ interchanged.
We conclude that \ref{ii} holds.

\ref{ii}$\Rightarrow$\ref{iii}:
It follows from \RL{basic}{a} that $Q\cdot Q^*=S\cdot S^*\cdot T\cdot T^*=A\cdot B$
and $A,B$ are quadratic by \RL{basic}{e}. Hence, \ref{iii} holds.

\ref{iii}$\Rightarrow$\ref{i}:
By definition,
\[
[Q]\cap\cN=\P(\set{Q(s,t)}{s,t\in\C \text{ \st } (Q\cdot Q^*)(s,t)=A(s)\cdot B(t)=0}).
\]
Hence, $[Q(s,p)]\subset\cN$ for $p\in\C$ \st $B(p)=0$.
Since $Q(s,p)$ has degree one, we find that $[Q(s,p)]$ is either
a complex line or a complex point.
Moreover, $[Q(s,p)]$ is non-real and thus $p\in\C\setminus\R$.
Similarly, $[Q(q,t)]\subset\cN$ is contained in a complex line
for some $q\in\C\setminus\R$ \st $A(q)=0$.
We conclude that \ref{i} holds.

We concluded the proof as all implications are reached.
\end{proof}

\section{A projective model for M\"obius geometry}
\label{sec:model}

We setup a projective model for M\"obius geometry
and introduce some conventions.

We define a \df{real variety}~$X$ to be a complex irreducible variety together with
an antiholomorphic involution $\sigma_X\c X\to X$
and we denote its real points by
\[
X_\R:=\set{p\in X}{\sigma_X(p)=p}.
\]
Such varieties can always be defined by polynomials with real coefficients.
In what follows,
points, lines, surfaces and projective spaces $\P^n$ are real algebraic varieties
unless explicitly stated otherwise.
A conic is real by default, but may be reducible.
We assume that the real structure $\sigma_{\P^n}\c\P^n\to\P^n$ sends $x$
to $(\overline{x_0}:\ldots:\overline{x_n})$, where $\overline{c}$ denotes the complex conjugate of $c\in \C$.
Maps between real varieties are compatible with $\sigma_X$ by default.

We denote the \df{singular locus} of $X$ by $\Sing X$.
We call a complex line $L\subset X$
a complex \df{double line} if it has multiplicity two in $X$ \citep[page~388]{1977}.

Let $f\c A \dto B\subset\P^n$ be a rational map that is not defined at~$U\subset A$.
We assume that the image~$B\subset\P^n$ is not contained in a hyperplane section.
By abuse of notation we denote $f(A\setminus U)\subseteq B$ by $f(A)$.
We call $f$ a \df{morphism} if it is everywhere defined and thus $U=\varnothing$.
The image of a morphism between projective varieties is Zariski closed \cite[Theorem~II.4.9]{1977}.
We call $f$ \df{biquadratic} if $A=\P^1\times\P^1$ and the components
of $f$ are of bidegree (2,2). Notice that if $f$ is biquadratic, then
its components are of total degree 4.
Components of rational maps are by default defined over $\R$.

The $n$-dimensional \df{M\"obius quadric} is defined as
\[
\S^n:=\set{x\in\P^{n+1}}{-x_0^2+x_1^2+\cdots+x_{n+1}^2=0}.
\]
Let $\aut_\C\S^3$ be defined as the set of complex linear automorphisms
$\varphi\c\P^4\to\P^4$ \st $\varphi(\S^3)=\S^3$.
The \df{M\"obius transformations} are defined as
\[
\aut\S^3:=\set{\varphi\in\aut_\C\S^3}{\varphi\circ\sigma_{\P^4}=\sigma_{\P^4}\circ\varphi}.
\]
The \df{elliptic absolute} is
defined as the smooth quadric
$\cE:=\set{x\in \S^3}{ x_0=0}$.
The \df{Euclidean absolute} is defined as
the quadratic cone $\cU:=\set{x\in \S^3}{x_0-x_4=0}$.
The \df{hyperbolic absolute} is defined as
the smooth quadric $\cH:=\set{x\in \S^3}{x_4=0}$.

The \df{stereographic projection} $\pi\c\S^3\dto\P^3$ sends $x$ to $(x_0-x_4:x_1:x_2:x_3)$.
The \df{central projection}~$\tau\c \S^3\to \P^3$
is the linear projection that sends $(x_0:\ldots:x_4)$
to $(x_1:\ldots:x_4)$. The map $\tau$ is a 2:1 morphism with center $(1:0:\ldots:0)$, ramification locus~$\cE$
and branching locus~$\tau(\cE)=\cN$.

Let $\iota_n\c \R^n\hookrightarrow \P^n$ send $(y_1,\ldots,y_n)$ to $(1:y_1:\ldots:y_n)$.
Its restricted rational inverse $\iota_n^{-1}\c\P_\R^n\dto\R^n$ sends
$(x_0:\ldots:x_n)$ to $(x_1,\ldots,x_n)/x_0$.
If $C\subset\R^n$, then $\bP(C)\subset\P^n$ denotes the Zariski closure of~$\iota_n(C)$.
If $C\subset\P^n$, then $\bR(C)\subset\R^n$ is defined as $\iota_n^{-1}(\set{x\in C_\R}{x_0\neq 0})$.
If $C\subset\R^3$, then $\bS(C)\subset\S^3$ denotes the Zariski closure of
the inverse stereographic projection~$\pi^{-1}(\bP(C))$.
For example, $\bP(S^3)=\S^3$, $\bR(\S^3)=S^3$ and $\bS(\R^3)=\S^3$.
If $\bP(\{a\})=\{b\}$ for $a\in\R^n$, then we write $\bP(a)=b$ instead.
Similarly, if $\bR(\{a\})=\{b\}$ for $a\in\P^n$, then we write $\bR(a)=b$ instead.

We call $X\subset \S^3$
\df{$\lambda$-circled},
\df{Darboux cyclide},
\df{ring cyclide},
\df{Perseus cyclide},
\df{Blum cyclide}, or
\df{CH1 cyclide},
if $X=\bS(Z)$ and $Z\subset\R^3$ is defined as such in \textsection\ref{sec:c}.
We call an irreducible conic $C\subset\S^3$ a circle if $\bR(C)\subset S^3$ is a circle.
We remark that surfaces $Z,Z'\subset \R^3$
are M\"obius equivalent if and only if
$\bS(Z)=\nu(\bS(Z'))$ for some M\"obius transformation~$\nu\in\aut\S^3$.

For all $\alpha\in\P^1$, we call
$\set{(\alpha,y)\in\P^1\times\P^1}{y\in\P^1}$
the \df{left fiber associated to $\alpha$}
and
$\set{(x,\alpha)\in\P^1\times\P^1}{x\in\P^1}$
the \df{right fiber associated to $\alpha$}.
We remark that a left fiber and right fiber are of bidegrees~$(1,0)$ and~$(0,1)$, \resp.
If $\varphi\c \P^1\times\P^1\dto X$ is a rational map
and $F$ is a left or right fiber,
then we call the Zariski closure of $\varphi(F)$
a \df{left image} and \df{right image}, \resp.
A left or right image is \df{associated to $\alpha$}
if its corresponding left or right fiber is as such.

A rational map $\varphi\c\P^1\times\P^1\to X\subset\S^3$
is $H$-\df{decomposable}
for $H\in\{\cE,\cU,\cH\}$
if there exist
complex lines $L,\oL,R,\oR\subset X$
\st
\[
X\cap H=L\cup\oL\cup R\cup \oR,
\]
where $L,\oL$ are complex conjugate left images
and $R,\oR$ are complex conjugate right images.

\section{Biquadratic birational morphisms}
\label{sec:morphism}

The goal of this section is to show \Cref{prp:EU},
namely that a biquadratic birational morphism
$\varphi\c\P^1\times\P^1\to X\subset \S^3$ is up to M\"obius equivalence
either $\cE$-decomposable or $\cU$-decomposable.

\begin{lemma}
\label{lem:Y}
Suppose that $\varphi\c\P^1\times\P^1\to X\subset\P^n$ is a biquadratic birational morphism
and let $F,F'\subset\P^1\times\P^1$ be complex curves of bidegrees $(a,b)$ and $(a',b')$,
\resp.
\begin{claims}
\item\label{lem:Y:A}
If $F$ and $F'$ meet transversally in finitely many points,
then $|F\cap F'|=a\cdot b'+a'\cdot b$.

\item\label{lem:Y:B}
If $\varphi(F)$ is a hyperplane section, then $(a,b)=(2,2)$.

\item\label{lem:Y:C}
$\deg X=8$.

\item\label{lem:Y:D}
$\varphi(F)$ is not a complex point.

\item\label{lem:Y:E}
If a general fiber of $\varphi|_F\c F\to \varphi(F)$ has cardinality $d\in\Z_{>0}$,
then
\[
\deg \varphi(F)=2\cdot\frac{a+b}{d}.
\]
\end{claims}
\end{lemma}

\begin{proof}
Well-known, but we include some details for the novice.
It follows from \citep[Exercise~I.2.15a and Example~V.1.9.2]{1977}
that the group~$N(\P^1\times\P^1)$ of divisor classes up to numerical equivalence
is isomorphic to
$\bas{\l_0,\l_1}_\Z$, where $\l_0^2=\l_1^2=0$ and $\l_0\cdot\l_1=1$.
Moreover, the classes of $F$ and $F'$ are
$[F]=a\,\l_0+b\,\l_1$ and $[F']=a'\,\l_0+b'\,\l_1$, \resp.

\ref{lem:Y:A}
It follows from \citep[Theorem~V.1.1]{1977} that
$|F\cap F'|=[F]\cdot [F']=a\cdot b'+a'\cdot b$.

\ref{lem:Y:B}
Let $(\varphi_0,\ldots,\varphi_n)$ denote the bidegree~(2,2) components of~$\varphi$.
Since $\varphi(F)$ is a hyperplane section,
we find that
$F=\set{(x,y)\in\P^1\times\P^1}{\alpha_0\,\varphi_0(x,y)+\cdots+\alpha_n\,\varphi_n(x,y)=0}$
for some $\alpha\in\P^n$. Hence, $(a,b)=(2,2)$ as was to be shown.

Let $H,H'\subset\P^n$ be general hyperplanes
and $G,G'\subset\P^1\times\P^1$ complex curves \st
$\varphi(G)=X\cap H$ and $\varphi(G')=X\cap H'$.
Notice that $[G]=[G']=2\,\l_0+2\,\l_1$ by Assertion~\ref{lem:Y:B}.

\ref{lem:Y:C}
By definition, $\deg X=|X\cap H\cap H'|$ (see \citep[\textsection I.7]{1977}).
Since $\varphi$ is a birational morphism and both $H$ and $H'$ are general,
it follows from Assertion~\ref{lem:Y:A} that $\deg X=|G\cap G'|=[G]\cdot [G']=8$.

\ref{lem:Y:D}
Suppose by contradiction that $\varphi(F)$ is a point.
In this case, $\varphi(F)\notin H$
and thus $F\cap G=\varnothing$.
We arrive at a contradiction
since $[F]\cdot [G]=0$ by Assertion~\ref{lem:Y:A} which implies that $[F]=0$.

\ref{lem:Y:E}
By definition, $\deg \varphi(F)=|\varphi(F)\cap H|$.
Since $H$ is general, we may assume that $F$ and $G$ meet transversally in finitely many points.
The map $\varphi|_F$ is a $d:1$ morphism and thus
\[
\deg \varphi(F)=\tfrac{1}{d}\cdot|F\cap G|=\tfrac{1}{d}\cdot[F]\cdot[G]=2\cdot\frac{a+b}{d}
\]
by Assertion~\ref{lem:Y:A}.
\end{proof}

\begin{lemma}
\label{lem:nopoint}
If $\varphi\c\P^1\times\P^1\to X\subset\S^3$ is a biquadratic birational morphism,
then the complex map~$\varphi(p,\_)\c\P^1\to C_p\subset X$ is for all $p\in \P^1$ either a
\begin{Mlist}
\item complex 1:1 morphism and $C_p$ is a complex irreducible conic, or
\item complex 2:1 morphism and $C_p$ is a complex double line in $X$.
\end{Mlist}
\end{lemma}

\begin{proof}
By definition, $C_p=\varphi(F)$ for the left fiber $F$ associated to $p$.
We know from \RL{Y}{D} that $\varphi(F)$ is not a point.
Since $F$ has bidegree~(0,1), it follows from \RL{Y}{E}
that $\deg C_p$ is equal to either 2 or 1 so that $\varphi|_F$ is a 1:1 and 2:1 morphism, \resp.
This concludes the proof, as $C_p$ must be an irreducible conic and complex double line, \resp.
\end{proof}

Recall from \Cref{sec:model} that
$\pi\c\S^3\dto\P^3$ denotes the stereographic projection
sending $x$ to $(x_0-x_4:x_1:x_2:x_3)$.

\begin{lemma}
\label{lem:L}
Suppose that $\varphi\c\P^1\times\P^1\to X\subset \S^3$ is a biquadratic birational morphism.
If there exist $p\in\P^1$
and
$u_0,\ldots,u_3,v_0,\ldots,v_3\in\C$
such that
\[
(\pi\circ\varphi)(p,y)=
(
u_0\,y_1^2+v_0\,y_0^2:
u_1\,y_1^2+v_1\,y_0^2:
u_2\,y_1^2+v_2\,y_0^2:
u_3\,y_1^2+v_3\,y_0^2
)
,
\]
then $\set{\varphi(p,y)}{y\in\P^1}$ is a complex line in~$X$.
\end{lemma}

\begin{proof}
Let $L:=\set{\varphi(p,y)}{y\in\P^1}\subset X$.
We observe that
\[
(\pi\circ\varphi)(p,(y_0:y_1))=(\pi\circ\varphi)(p,(y_0:-y_1))
\]
and thus $(\pi\circ\varphi)(p,\_)\c \P^1\dto \pi(L)$
is either $2:1$ or $\infty:1$.
This implies that $\pi(L)$ is a complex line and a complex point, \resp.
It follows from \Cref{lem:nopoint} that $L$ is
either a complex irreducible conic, or a complex line.
Suppose by contradiction that $L$ is a complex irreducible conic.
In this case $\pi(L)$ must be a complex line and thus $L$ lies in a complex plane
that passes through the center of stereographic projection $\pi\c \S^3\dto \P^3$.
Let $q\in\pi(L)$ be general so that $\pi^{-1}(q)$ consists of two complex points~$r,s\in L$.
The complex line through $r$ and $s$ meets the center of $\pi$ and
is therefore contained in $\S^3$ by B\'ezout's theorem.
All the complex lines in~$\S^3$ that pass through
the center of $\pi$ lie in the Euclidean absolute $\cU$.
We arrived at a contradiction as $\cU$ does not contain complex planes.
It follows that the preimage $L$ of $\pi(L)$
must be a complex line as was to be shown.
\end{proof}

\begin{lemma}
\label{lem:coord}
If $\varphi\c\P^1\times\P^1\dto\P^{n+1}$ is a biquadratic rational map,
then there exists $p\in\P^1$ and $\nu\in\aut_\C\S^n$ \st
\[
(\pi\circ\nu\circ \varphi)(p,y)=(u_0\,y_1^2+v_0\,y_0^2:\cdots:u_n\,y_1^2+v_n\,y_0^2),
\]
where $u_i,v_i\in \C$ for all $0\leq i\leq n+1$.
\end{lemma}

\begin{proof}
Since $\varphi=(\varphi_0:\ldots:\varphi_{n+1})$ is biquadratic,
there exist quadratic forms $\alpha_i,\beta_i,\gamma_i\in \C[x_0,x_1]$ \st
for all $i\leq 0\leq n+1$,
\[
\varphi_i(x,y)=\alpha_i(x)\,y_1^2+\beta_i(x)\,y_0y_1 + \gamma_i(x)\,y_0^2.
\]
We assume that $(\beta_0(p),\ldots,\beta_{n+1}(p))\neq (0,\ldots,0)$ for all $p\in\P^1$,
otherwise there is nothing to proof.
Therefore,
$\beta(x):=(\beta_0(x):\ldots:\beta_{n+1}(x))$ parametrizes a curve~$C\subset \P^{n+1}$.
By B\'ezout's theorem, there exists $p\in\S^n\cap C$, which implies that
\[
-\beta_0(p)^2+\beta_1(p)^2+\cdots+\beta_{n+1}(p)^2=0.
\]
Hence, there exists a complex M\"obius transformation $\nu\in\aut_\C\S^n$
\st
\[
\nu(\beta(p))=(1:0:\cdots:0:1).
\]
We conclude that $(\pi\circ\nu\circ \varphi)(p,y)$ is as asserted.
\end{proof}

\begin{lemma}
\label{lem:line}
If $\varphi\c\P^1\times\P^1\to X\subset \S^3$ is a biquadratic birational morphism,
then there exists $p\in\P^1$ such that $\set{\varphi(p,y)}{y\in\P^1}$ is
a complex line in~$X$.
\end{lemma}

\begin{proof}
We know from \Cref{lem:coord} that $\varphi$ satisfies the hypothesis
of \Cref{lem:L} up to compositions with complex automorphisms in $\aut_\C\S^3$.
As elements of $\aut_\C\S^3$ send complex lines to complex lines
the proof is concluded by \Cref{lem:L}.
\end{proof}

\begin{lemma}
\label{lem:S3}
If $X\subset\S^3$ is a surface that contains complex lines $L,R\subset X$
\st $|L\cap R|=1$, then
there exists a M\"obius transformation $\nu\in\aut\S^3$ and
hyperplane section $H\in\{\cE,\cU,\cH\}$
\st
\[
L\cup\oL\cup R\cup\oR\subseteq X\cap \nu(H),
\]
where $\oL=\sigma_X(L)$ and $\oR=\sigma_X(R)$.
\end{lemma}

\begin{proof}
Since $\sigma_X(X)=X$ and $|L\cap R|=1$,
it follows that
$L\cup\oL\cup R\cup\oR$
is contained in a real hyperplane section $H\subset \S^3$.
We observe that either $H=\cE$ or there exists $p\in\R^4$ \st
$H=\bP(Z_p)\cap\S^3$ and $Z_p\subset \R^4$ is the polar hyperplane
of $p$ with respect to $S^3\subset\R^4$.
We deduce that if $p$ lies either
in the interior of $S^3$, on $S^3$, or in the exterior of $S^3$,
then $H$ is M\"obius equivalent to the elliptic absolute $\cE$,
Euclidean absolute $\cU$ and hyperbolic absolute $\cH$, \resp.
\end{proof}

\begin{proposition}
\label{prp:EU}
If $\varphi\c\P^1\times\P^1\to X\subset \S^3$ is a biquadratic birational morphism,
then there exists a M\"obius transformation $\nu\in\aut\S^3$
\st $\nu\circ\varphi$ is either $\cE$-decomposable or $\cU$-decomposable.
\end{proposition}

\begin{proof}
Let $\theta\c\P^1\times\P^1\to\P^1\times\P^1$ send $(x,y)$ to $(y,x)$.
It follows from \Cref{lem:line} applied to $\varphi$ and $\varphi\circ\theta$
that there exist complex conjugate left fibers $A,\oA\subset\P^1\times\P^1$
and complex conjugate right fibers $B,\oB\subset\P^1\times\P^1$
\st the left images $L=\varphi(A)$, $\oL=\varphi(\oA)$
and right images
$R=\varphi(B)$, $\oR=\varphi(\oB)$
are complex lines in $X$ \st $\sigma_X(R)=\oR$ and $\sigma_X(L)=\oL$.
Since $\S^3$ does not contain lines, we observe that $L\neq\oL$
and $R\neq\oR$.
Recall from \RL{Y}{A} that
$|A\cap B|=|\oA\cap B|=|A\cap\oB|=|\oA\cap\oB|=1$
and $|A\cap\oA|=|B\cap\oB|=0$.
By B\'ezout's theorem, $|L\cap R|\in\{0,1,\infty\}$.
Since $A\cap B\neq\varnothing$, we find that either $|L\cap R|=1$ or $L=R$.

Assume by contradiction that $L=R$. As $\sigma_X(X)=X$ and $L\neq\oL$,
we find that $\oL=\oR$ and $L\cap\oL=\{p\}$ for some $p\in X_\R$.
We may assume \Wlog that $p$ is the center of
the stereographic projection~$\pi\c\S^3\dto\P^3$
so that $\pi(L)$ and $\pi(\oL)$ are complex conjugate points.
Let $C\subset X$ be a real general left image and
$C'\subset X$ be a real general right image so that $q\in C\cap C'$ for some $q\in X\setminus(L\cup\oL)$.
We find that $C$ and $C'$ are irreducible conics
\st $D\cap M\neq\varnothing$ for all $D\in\{C,C'\}$ and $M\in\{L,\oL\}$.
Hence, $\pi(C),\pi(C')\subset\P^3$ are irreducible conics
\st $\{\pi(q),\pi(L),\pi(\oL)\}\subset\pi(C)\cap\pi(C')$. This is implies
that $\pi(C)$ and $\pi(C')$ are coplanar.
Since $C$ and $C'$ are general and independent it follows that
$\pi(X)$ contains a plane. We arrived at a contradiction as $X$ is irreducible and not a sphere.

Repeating the same argument for $(\oL,R)$, $(L,\oR)$ and $(\oL,\oR)$
instead of $(L,R)$,
we find that $|L\cap R|=|\oL\cap R|=|L\cap\oR|=|\oL\cap\oR|=1$.
By \Cref{lem:S3}, there exists $\nu\in\aut\S^3$ \st
$L\cup\oL\cup R\cup\oR\subseteq X\cap\nu(H)$ for some $H\in\{\cE,\cU,\cH\}$.
We know from \Cref{lem:nopoint} that $L,\oL,R,\oR\subset X$ are double lines.
Since $\deg X=8$ by \RL{Y}{C}, it follows from
B\'ezout's theorem that $L\cup\oL\cup R\cup\oR=\nu(H)$.
Now suppose by contradiction that $\nu(H)=\cH$.
By definition, the hyperbolic absolute $\cH$ is equal to $\S^3$ and thus $|L\cap\oL|=|R\cap\oR|=1$.
Since $|L\cap R|=|\oL\cap\oR|=1$, this implies that $L\cap\oR=\varnothing$.
We arrived at a contradiction as $|L\cap\oR|=1$.
Hence, $\nu\circ\varphi$ is either $\cE$-decomposable or $\cU$-decomposable
as was to be shown.
\end{proof}

\section{Biquadratic \texorpdfstring{$\cE$}{E}-decomposable morphisms}
\label{sec:E}

Suppose we are given a biquadratic birational morphism
$\varphi\c\P^1\times\P^1\to X\subset \S^3$.
The goal of this section is to show \Cref{prp:ECD},
namely if $\varphi$ is $\cE$-decomposable,
then there exist circles $C,D\subset S^3$ \st $X=\bP(C\cdot D)$.

\begin{lemma}
\label{lem:ST}
Suppose that $S,\tS\in\H[s]$ and $T,\tT\in\H[t]$ are linear \st
\[
S\cdot S^*=\tS\cdot\tS^*=u\cdot (s-a)\cdot(s-\oa)
\quad\text{and}\quad
T\cdot T^*=\tT\cdot\tT^*=v\cdot (t-b)\cdot(t-\ob)
\]
for some $a,b\in\C$ and $u,v\in\R$.
\begin{claims}
\item\label{lem:ST:a} $[T(b)\cdot S \cdot\tS \cdot \tT(b)]\in \cN$ is a complex point.
\item\label{lem:ST:b} There exists $c\in\C$ \st $S(a)\cdot \tT(c)\cdot \tS(a)\cdot T(c)=0$.
\item\label{lem:ST:c} If $S(a)\cdot\tT(b)\cdot\tS(a)\cdot T(b)=0$
and
$[S\cdot\tT(b)\cdot\tS\cdot T(b)]$,
$[S(a)\cdot\tT\cdot\tS(a)\cdot T]$
are both not complex points,
then there exist linear $T'\in\H[t]$ and $S'\in\H[s]$ \st
\[
S\cdot\tT\cdot\tS\cdot T=S\cdot S'\cdot T'\cdot T,
\]
where $S'\cdot S'^*=\tS\cdot\tS^*$ and $T'\cdot T'^*=\tT\cdot\tT^*$.
\end{claims}
\end{lemma}

\begin{proof}
\ref{lem:ST:a}
It follows from \Cref{cor:fct} that
$[T(b)\cdot S\cdot\tS\cdot\tT(b)]$
is contained in both a left and right generator.
By \RL{N}{a} these generators meet in a single complex point
and thus Assertion~\ref{lem:ST:a} follows.

\ref{lem:ST:b}
It follows from \Cref{cor:fct}
that $[S(a)\cdot\tT\cdot\tS(a)]$ is
contained in both a left and right generator
and thus a complex point by \RL{N}{a}.
It follows from \RL{basic}{c} that
$t-c \mid S(a)\cdot\tT\cdot\tS(a)$
for some $c\in\C$ so that
Assertion~\ref{lem:ST:b} follows.

\ref{lem:ST:c}
We know from the the abc-\Cref{lem:abc} that either
\[
S(a)\cdot\tT(b)=0,\quad
\tT(b)\cdot\tS(a)=0 \quad\text{or}\quad
\tS(a)\cdot T(b)=0.
\]
For example, if $S(a)\cdot\tT(b)=0$, then $s-a \mid S\cdot\tT(b)$ by \RL{basic}{b}
and thus $[S\cdot\tT(b)]$ is a complex point by \RL{basic}{c}.
This complex point in~$\cN$ is by \RL{N}{a} the intersection of some left and right generator.
The other two cases are analogous and thus either
\[
[S\cdot\tT(b)]\in\cL_\alpha\cap\cR_\beta,\quad
[\tT(b)\cdot\tS]\in\cL_\alpha\cap\cR_\beta\quad\text{or}\quad
[\tS(a)\cdot T]\in\cL_\alpha\cap\cR_\beta
\]
for some $\alpha,\beta\in\HC$ \st $[\alpha^*],[\beta^*]\in\cN$.

Suppose by contradiction that $[S\cdot\tT(b)]\in\cL_\alpha\cap\cR_\beta$.
It follows from \Cref{cor:fct} that
$S\cdot \tT=T'\cdot S'$ for some linear $S'\in\H[s]$ and $T'\in\H[t]$ \st $S'\cdot S'^*=S\cdot S^*$ and
$T'\cdot T'^*=\tT\cdot \tT^*$.
This implies that $S\cdot\tT(b)\cdot\tS\cdot T(b)=T'(b)\cdot S'\cdot\tS\cdot T(b)$.
Hence, $[S\cdot T(b)\cdot \tS\cdot \tT(b)]$ is a complex point by Assertion~\ref{lem:ST:a}
and thus we arrived at a contradiction.

Suppose by contradiction that $[\tS(a)\cdot T]\in\cL_\alpha\cap\cR_\beta$.
It follows from \Cref{cor:fct} that
$\tS\cdot T=T'\cdot S'$ for some linear $S'\in\H[s]$ and $T'\in\H[t]$ \st $S'\cdot S'^*=\tS\cdot\tS^*$ and
$T'\cdot T'^*=T\cdot T^*$.
This implies that $S(a)\cdot \tT\cdot \tS(a)\cdot T=S(a)\cdot \tT\cdot T'\cdot S'(a)$.
Hence, $[S(a)\cdot \tT\cdot  \tS(a)\cdot T]$ is a complex point by Assertion~\ref{lem:ST:a}
and thus we arrived at a contradiction.

Finally, suppose that $[\tT(b)\cdot\tS]\in\cL_\alpha\cap\cR_\beta$.
It follows from \Cref{cor:fct} that
$\tT\cdot\tS=S'\cdot T'$ for some linear $S'\in\H[s]$ and $T'\in\H[t]$ \st $S'\cdot S'^*=\tS\cdot \tS^*$ and
$T'\cdot T'^*=\tT\cdot\tT^*$.
Hence, $S\cdot \tT\cdot \tS\cdot T=S\cdot S'\cdot T'\cdot T$ as was to be shown.

We considered all three cases and thus concluded the proof.
\end{proof}

\begin{remark}
If the linear polynomials $S$, $\tS$, $T$ and $\tT$ at \RL{ST}{c} are monic,
then it follows from the proofs of \citep[Theorems~7.3 and 7.6]{2022n1}
that $S'=\tS$ and $T'=\tT$ so that $\tT\cdot \tS=\tS\cdot\tT$.
\END
\end{remark}

Suppose that $\varphi\c \P^1\times\P^1\dto X\subset\S^3$ is a rational map with
components $(\varphi_0,\ldots,\varphi_4)$.
The \df{associated quaternionic polynomial} $Q\in\H[s,t]$ of~$\varphi$ is
up to scalar multiplication defined as
\[
Q:=Q_1+Q_2\,\qi+Q_3\,\qj+Q_4\,\qk,
\]
where $Q_i(s,t):=\varphi_i(1:s;1:t)$ for $0\leq i\leq 4$.
Notice that $Q$ is primitive as the components of $\varphi$ do not have a common factor.
Since $X\subset\S^3$, we deduce that
\[
Q\cdot Q^*=Q_1^2+\cdots+Q_4^2=Q_0^2.
\]
Recall from \textsection\ref{sec:model} that the central projection~$\tau\c \S^3\to \P^3$ is a 2:1 morphism with
center~$(1:0:0:0:0)$, ramification locus~$\cE$
and branching locus~$\tau(\cE)=\cN$.

Let $\theta\c\P^1\times\P^1\to\P^1\times\P^1$ send $(x,y)$ to $(y,x)$.

\begin{lemma}
\label{lem:Q}
If $\varphi\c\P^1\times\P^1\dto X\subset \S^3$ is an
$\cE$-decomposable biquadratic birational map,
then
there exist linear polynomials $S,\tS\in\H[s]$ and $T,\tT\in\H[t]$
such that the associated quaternionic polynomial~$Q$ of either $\varphi$ or $\varphi\circ\theta$
is equal to either
$S\cdot \tT\cdot \tS\cdot T$ or $S\cdot\tS\cdot \tT\cdot T$.
\end{lemma}

\begin{proof}
Notice that $[Q]\cap\cN=\tau(X\cap\cE)$ consists
of two complex conjugate left generators and
two complex conjugate right generators.
Since $Q$ is associated to either $\varphi$ or $\varphi\circ\theta$,
we may assume \Wlog that the central projection of a right/left image in~$\cE$
is a left/right generator in~$\tau(\cE)$.
Hence, there exists $p,q\in\C$ and $[\alpha^*],[\beta^*]\in\cN$ \st
$Q(s,p)=\cL_\alpha$ and $Q(q,t)=\cR_\beta$.
It follows from \Cref{cor:fct}\ref{cor:fct:b} that
$Q=S\cdot\tQ$ for some linear $S\in\H[s]$ and $\tQ\in\H[s,t]$ of bidegree~$(1,2)$.
As $[Q(s,p)]=[S(s)\cdot \tQ(s,p)]=\cL_\alpha$ it follows that $S(s)\cdot \tQ(s,p)\cdot \alpha=0$ by definition.
Since $S\cdot S^*\neq 0$, this implies that $\tQ(s,p)\cdot \alpha=0$
and thus $[\tQ(s,p)]=\cL_\alpha$.
We apply \Cref{cor:fct}\ref{cor:fct:a} and find that $Q=S\cdot E\cdot T$,
where $T\in \H[t]$ is linear and $E\in\H[s,t]$ is of bidegree~$(1,1)$.
We apply \RL{basic}{a} and deduce that
$
Q\cdot Q^*=Q_0^2=S\cdot S^*\cdot T\cdot T^*\cdot E\cdot E^*.
$
This implies that $E\cdot E^*$ is up to scalar multiplication equal to the product
of quadratic polynomials $S\cdot S^*\in\R[s]$ and $T\cdot T^*\in\R[t]$.
Hence,
$E$ equals either $\tS\cdot\tT$ or $\tT\cdot\tS$ for some linear polynomials $\tS\in\H[s]$ and $\tT\in\H[t]$
by \Cref{cor:clifford}.
This concludes the proof.
\end{proof}

\begin{lemma}
\label{lem:SSTT}
If $\varphi\c\P^1\times\P^1\dto X\subset \S^3$ is a
biquadratic birational map \st its associated quaternionic polynomial~$Q$
is equal to $S\cdot\tS\cdot\tT\cdot T$ for some linear polynomials $S,\tS\in\H[s]$ and $T,\tT\in\H[t]$,
then there exist circles $C,D\subset S^3$ \st $X=\bP(C\cdot D)$.
\end{lemma}

\begin{proof}
By assumption $X\subset\S^3$ and thus $Q\cdot Q^*=Q_1^2+\cdots+Q_4^2=Q_0^2$,
where $Q=Q_1+Q_2\,\qi+Q_3\,\qj+Q_4\,\qk$.
Since $Q\cdot Q^*=Q_0^2$ and $Q$ is only defined up to scalar multiplication,
we can assume \Wlog that $S\cdot S^*=\tS\cdot\tS^*$ and $T\cdot T^*=\tT\cdot\tT^*$ (see \RL{basic}{a}).
This implies that $f(s):=S\cdot\tS/S\cdot S^*$ and
$g(t):=T\cdot\tT/T\cdot T^*$
are rational functions from $\R$ to $S^3$,
where we identified $S^3$ with the unit-quaternions.
Notice that $Q/\sqrt{Q\cdot Q^*}=f(s)\cdot g(t)$ is a birational map from~$\R^2$ to~$\bR(X)\subset S^3$.
As the numerators and denominators of $f$ and $g$ are of degree two,
we deduce that these rational maps parametrize the circles~$C\subset S^3$ and~$D\subset S^3$, \resp.
This concludes the proof as $X=\bP(C\cdot D)$.
\end{proof}

\begin{proposition}
\label{prp:ECD}
If $\varphi\c\P^1\times\P^1\to X\subset \S^3$
is a $\cE$-decomposable biquadratic birational morphism,
then there exist circles $C,D\subset S^3$ \st $X=\bP(C\cdot D)$.
\end{proposition}

\begin{proof}
Let $(\varphi_0,\ldots,\varphi_4)$ denote the components of $\varphi$
and suppose by contradiction that $Q=S\cdot\tT\cdot \tS\cdot T$
is the associated quaternionic polynomial of $\varphi$ for some linear polynomials
$S,\tS\in\H[s]$ and $T,\tT\in\H[t]$.
It follows from \RL{ST}{b} that there exist $a,c\in \C$
\st $Q(a,c)=0$, which implies that $\varphi_0(r)=\cdots=\varphi_4(r)=0$ for $r=(1:a;1:c)$.
We arrived at a contradiction as $\varphi$ is a morphism and thus
everywhere defined.
The proof is now concluded by Lemmas~\ref{lem:Q} and~\ref{lem:SSTT}.
\end{proof}

\section{Biquadratic \texorpdfstring{$\cU$}{U}-decomposable morphisms}
\label{sec:U}

Suppose that $\varphi\c\P^1\times\P^1\to X\subset \S^3$ is a biquadratic birational morphism.
The goal of the current section is to show \Cref{prp:UAB},
which states that if $\varphi$ is $\cU$-decomposable, then
there exist circles $A,B\subset\R^3$ \st $X=\bS(A+B)$.

Recall from \textsection\ref{sec:model} that an irreducible conic $C\subset\S^3$
is a circle if $\bR(C)\subset\S^3$ is a circle.
Notice that the following lemma remains true if we interchange ``left'' and ``right''.

\begin{lemma}
\label{lem:C}
Suppose that $\varphi\c\P^1\times\P^1\to X\subset\S^3$ is a biquadratic birational morphism.
\begin{claims}
\item\label{lem:C:a}
If $C\subset X$ is a circle \st $C\nsubseteq\Sing X$,
then $C$ is a real left or right image.
If $C$ is a real left or right image,
then $C$ is a circle.

\item\label{lem:C:b}
If $C_\alpha$ is a left image associated to $\alpha\in\P^1$
\st either $\alpha$ is general or $\varphi^{-1}(C_\alpha)$ does not contain two complex curves,
then there exists a finite subset $P\subset\P^1$
\st for all $\beta\in\P^1\setminus P$,
\[
|C_\alpha\cap C_\beta|=
\begin{cases}
0 & \text{if $C_\beta$ is the left image associated to $\beta$, and}\\
1 & \text{if $C_\beta$ is the right image associated to $\beta$}.
\end{cases}
\]
\end{claims}
\end{lemma}

\begin{proof}
\ref{lem:C:a}
Suppose that $F\subset\P^1\times\P^1$ is the curve \st $C=\varphi(F)$.
Since $\varphi$ is birational and $C\nsubseteq\Sing X$, we find that
the general fiber of the restriction $\varphi|_F$ has cardinality one.
If $C$ is a circle, then $F$ has either bidegree~$(1,0)$
or~$(0,1)$ by \RL{Y}{E} so that $C$ is a real left image and real right image, \resp.
Conversely, if $C$ is a real left or right image, then
$|C_\R|=\infty$ and $C$ is either a circle or double line by \Cref{lem:nopoint}.
A complex line in $\S^3$ must be non-real and thus $C$ is a circle as claimed.

\ref{lem:C:b}
By assumption, $\varphi$ is a birational morphism between projective varieties
and thus there exists a possibly reducible
curve $G\subset\P^1\times\P^1$ \st
$\P^1\times\P^1\setminus G\cong X\setminus\varphi(G)$.
By definition, $C_\alpha$ is the image of the left fiber $F_\alpha\subset\P^1\times\P^1$
and $C_\beta$ is the image of either the left or right general fiber $F_\beta\subset\P^1\times\P^1$.
By \RL{Y}{D}, both $C_\alpha$ and $C_\beta$ are complex curves.
If $\alpha$ is general, then $F_\alpha\not\subset G$
and thus $\varphi^{-1}(C_\alpha)\setminus F_\alpha$ is finite.
If $\alpha$ is not general, then
$\varphi^{-1}(C_\alpha)\setminus F_\alpha$ is finite by the assumption
that the preimage of $C_\alpha$ does not contain two complex curves.

First suppose that $C_\beta$ is a left image.
Let us assume by contradiction that $|C_\alpha\cap C_\beta|>0$.
Since $|F_\alpha\cap F_\beta|=0$,
there exist complex curves $A,B\subset G$
\st
$a\in F_\alpha\cap A$, $b\in F_\beta\cap B$
and $\varphi(a)=\varphi(b)\in C_\alpha\cap C_\beta$.
We know from \RL{Y}{D} that neither $\varphi(A)$ nor $\varphi(B)$
consists of a single complex point.
Moreover, $\varphi(B)\neq C_\alpha$
since
$B\neq F_\alpha$ and
$\varphi^{-1}(C_\alpha)\setminus F_\alpha$ is finite.
By assumption $\beta$ is general and
thus we deduce that $|C_\alpha\cap C_\beta|=0$ as asserted.

Finally suppose that $C_\beta$ is a right image.
Since $|F_\alpha\cap F_\beta|=1$, we find that $|C_\alpha\cap C_\beta|\geq 1$.
Now assume by contradiction that $|C_\alpha\cap C_\beta|\neq 1$.
In this case there exist complex curves $A,B\subset G$
\st
$a\in F_\alpha\cap A$, $b\in F_\beta\cap B$
and $\varphi(a)=\varphi(b)\in C_\alpha\cap C_\beta$.
The same argument as before
leads to a contradiction
and therefore $|C_\alpha\cap C_\beta|=1$ as asserted.
\end{proof}

We call two circles in $\R^3$ \df{parallel} if their spanning planes are parallel.

\begin{lemma}
\label{lem:LLpar}
If $C,C'\subset\S^3$ are circles and
there exist complex conjugate lines $R,\oR\subset\cU$
\st
$|C\cap R|=|C'\cap\oR|=1$
and
$(1:0:0:0:1)\notin C\cup C'$,
then
$\bR(\pi(C))$ and $\bR(\pi(C'))$ are parallel circles in~$\R^3$.
\end{lemma}

\begin{proof}
We observe that~$\pi(C)$ and~$\pi(C')$ are irreducible conics with infinitely many real points
that both meet the conic~$\pi(\cU)=\set{y\in\P^3}{y_0=y_1^2+y_2^2+y_3^2=0}$
in the complex conjugate points $\pi(R)$ and $\pi(\oR)$.
Suppose that $S_c\subset\R^3$ is a sphere with radius $c_0$ and center $(c_1,c_2,c_3)$ so that
\[
\bP(S_c):=\set{y\in\P^3}{-(c_0\,y_0)^2+(y_1-c_1\,y_0)^2+(y_2-c_2\,y_0)^2+(y_3-c_3\,y_0)^2=0}.
\]
We deduce that there exists $c,c'\in\R^3$ and hyperplanes~$H,H'\subset\R^3$
\st $\pi(C)=\bP(S_c\cap H)$ and $\pi(C')=\bP(S_{c'}\cap H')$.
Hence, the stereographic projections $\bR(\pi(C))$ and $\bR(\pi(C'))$ are circles in~$\R^3$.
Since $\bP(H)\cap\bP(H')$ is the line spanned by $\pi(R)$ and $\pi(\oR)$ in the plane at infinity,
we find that $H$ and $H'$ are parallel planes.
This concludes the proof as $\bR(\pi(C))=S_c\cap H$ and $\bR(\pi(C'))=S_{c'}\cap H'$.
\end{proof}

\begin{proposition}
\label{prp:UAB}
If $\varphi\c\P^1\times\P^1\to X\subset \S^3$ is a $\cU$-decomposable biquadratic birational morphism,
then there exist circles $A,B\subset\R^3$ \st
$X=\bS(A+B)$.
\end{proposition}

\begin{proof}
By definition, $X\cap\cU$ consists of two complex conjugate left images $L$, $\oL$
and two complex conjugate right images $R$, $\oR$ that meet at the center
$(1:0:0:0:1)$ of stereographic projection $\pi\c\S^3\dto\P^3$.
Notice that $\pi(\cU)=\set{z\in\P^3}{z_0=z_1^2+z_2^2+z_3^2=0}$ and
let $p,\op,q,\oq\in\pi(\cU)$ \st $\pi(R)=\{p\}$, $\pi(\oR)=\{\op\}$,  $\pi(L)=\{q\}$ and $\pi(\oL)=\{\oq\}$.
By \RL{Y}{B} the preimage of $\cU$ consists of two left fibers and two right fibers
and thus the hypothesis of \RL{C}{b} is satisfied.
We deduce from \RLS{C}{a}{C}{b} that
for almost all $u\in\P^1_\R$ the real left image $C_u\subset X$
is a circle that meets both~$R$ and~$\oR$ outside $R\cap \oR$.
Hence, its stereographic projection $\pi(C_u)$ meets the conic $\pi(\cU)$
in the two complex conjugate points $p$ and $\op$.
Similarly, for almost all $v\in\P^1_\R$ the real right image
$C'_v\subset X$ is a circle and its projection $\pi(C'_v)$ meets $\pi(\cU)$ in
the complex conjugate points $q$ and $\oq$.

Let $Y\subset\P^3$ denote the Zariski closure of $\pi(X)$
and consider the birational biquadratic map
$f:=\pi\circ\varphi\c\P^1\times\P^1\dto Y$.
As the complex lines in $\cU$ are non-real,
there exists non-real $\alpha,\beta\in\P^1$ \st
$\varphi(\alpha)=\varphi(\oalpha)=\varphi(\beta)=\varphi(\obeta)=(1:0:0:0:1)$.
We deduce that $f(\alpha,y)=p$ and $f\left(\oalpha,y\right)=\op$
for all $y\in\P^1\setminus\{\beta,\obeta\}$.
This implies that there exist forms $a_0,a_1,a_2,a_3\in\R[x]$ and $b_0,b_1,b_2,b_3\in\R[x,y]$
\st $a_0(\alpha)=a_0(\oalpha)=0$ and
\begin{align*}
f(x,y)=(a_0(x)\,b_0(x,y)&:a_1(x)\,b_0(x,y)+a_0(x)\,b_1(x,y)\\
                        &:a_2(x)\,b_0(x,y)+a_0(x)\,b_2(x,y)\\
                        &:a_3(x)\,b_0(x,y)+a_0(x)\,b_3(x,y)).
\end{align*}
In particular, notice that $f(\alpha,y)=(a_0(\alpha):a_1(\alpha):a_2(\alpha):a_3(\alpha))=p$
for all $y\in\P^1\setminus\{\beta,\obeta\}$.
Similarly, $f(x,\beta)=q$ and $f\left(x,\obeta\right)=\oq$ for all $x\in\P^1\setminus\{\alpha,\oalpha\}$.
Since the components of $f$ are of bidegree $(2,2)$ the same argument as before shows that
$b_0$, $b_1$, $b_2$, $b_3$ are quadratic forms in $\R[y]$ and thus
\begin{multline*}
f(x,y)=(a_0(x)\,b_0(y):
a_1(x)\,b_0(y)+a_0(x)\,b_1(y):
a_2(x)\,b_0(y)+a_0(x)\,b_2(y):\\
a_3(x)\,b_0(y)+a_0(x)\,b_3(y)).
\end{multline*}
It follows from \Cref{lem:LLpar} that there exist circles $A,B\subset \R^3$
and birational morphisms $a\c \P^1\to\bP(A)$ and $b\c\P^1\to\bP(B)$
\st
\[
a(x)=(a_0(x):\cdots:a_3(x))
\quad\text{and}\quad
b(y)=(b_0(y):\cdots:b_3(y)).
\]
Since $a_i/a_0+b_i/b_0=(a_i\,b_0+a_0\,b_i)/a_0\,b_0$ for $i\in\{1,2,3\}$ and
\[
\set{
\left(\frac{a_1}{a_0},\frac{a_2}{a_0},\frac{a_3}{a_0}\right)(1:s)+
\left(\frac{b_1}{b_0},\frac{b_2}{b_0},\frac{b_3}{b_0}\right)(1:t)
}{s,t\in\R}
\]
is Zariski dense in $A+B$, it follows that $Y=\bP(A+B)$.
We concluded the proof as $f=\pi\circ\varphi$ and thus $X=\bS(A+B)$ by definition.
\end{proof}

\section{The proof of Theorem II}
\label{sec:proof}

In this section, we conclude the proofs of \Cref{thm:c} and \Cref{cor:c},
which depend on the ``external'' Theorems~\ref{thm:A} and \ref{thm:B} stated below.
The proof of \Cref{thm:c}\ref{thm:c:d} uses \Cref{prp:ABCD}
in Appendix~\ref{app:ABCD}, which is self-contained.

\begin{theoremext}
\label{thm:A}
Suppose that $X\subset\S^3$ is a $\lambda$-circled surface
of degree $d$ \st $\lambda\geq 2$.
\begin{claims}
\item\label{thm:A:a}
$d\in\{2,4,8\}$ and if $d$ equals 4 or 8,
then $\lambda\leq 6$ and $\lambda=2$, \resp.
\item\label{thm:A:b}
If $d=8$, then there exists a biquadratic birational morphism $\P^1\times\P^1\to X$.
\end{claims}
\end{theoremext}

\begin{proof}
See \citep[Theorem~1]{2021circ} and \citep[Theorems~5 and 8]{2001}.
\end{proof}

\begin{theoremext}
\label{thm:B}
Suppose that $X\subset\S^3$ is a celestial surface of degree~$d\in\{2,4\}$.
\begin{claims}
\item\label{thm:B:a}
If $X=\bS(A+B)$ for some circles $A,B\subset\R^3$, then $d=2$.
\item\label{thm:B:b}
If $X=\bP(C\cdot D)$ for some circles $C,D\subset S^3$,
then $d=4$ and $X$ is either a Perseus cyclide, ring cyclide or CH1 cyclide.
\item\label{thm:B:c}
If $X$ is a ring cyclide, then there exists a M\"obius transformation $\nu\in\aut\S^3$
such that $\nu(X)=\bP(C\cdot D)$ for some circles $C,D\subset S^3$.
\item\label{thm:B:d}
If $d=4$ and $\bR(X)$ is smooth, then $X$ is $\lambda$-circled \st $\lambda\in\{2,4,5,6\}$.
\item\label{thm:B:e}
If $X$ is a Perseus cyclide or CH1 cyclide,
then there exists
a M\"obius transformation~$\nu\in\aut\S^3$ and
an $\cE$-decomposable birational map $\varphi\c\P^1\times\P^1\dto \nu(X)$
with components $(\varphi_0,\ldots,\varphi_4)$ of bidegree (2,2)
\st
\[
\Lambda:=\set{r\in\P^1\times\P^1}{\varphi_0(r)=\cdots=\varphi_4(r)=0}=\{(p,q),(\op,\oq)\},
\]
and
$X\cap\cE=L_p\cup L_\op\cup R_q\cup R_\oq$,
where
$L_p$ and $L_\op$ are complex lines and left images associated to $p$ and $\op$, \resp, and
$R_q$ and $R_\oq$ are complex lines and right images associated to $q$ and $\oq$, \resp.
\end{claims}
\end{theoremext}

See Appendix~\ref{app:B} for the proof of \Cref{thm:B}.

The following proposition confirms \citep[Conjecture~1]{2023trans}.

\begin{proposition}
\label{prp:E4}
Suppose that $X\subset\S^3$ is a celestial surface of degree 2 or 4.
Then $\nu(X)=\bP(C\cdot D)$ for some circles $C,D\subset S^3$ and a
M\"obius transformation $\nu\in\aut\S^3$
if and only if
$X$ is either a Perseus cyclide, ring cyclide or CH1 cyclide.
\end{proposition}

\begin{proof}
The $\Rightarrow$ direction follows from \Cref{thm:B}\ref{thm:B:b}.
We proceed with the $\Leftarrow$ direction.
If $X$ is a ring cyclide, then the proof is concluded by \Cref{thm:B}\ref{thm:B:c}.
Suppose that $X$ is a Perseus cyclide or CH1 cyclide.
It follows from \Cref{thm:B}\ref{thm:B:e} that there exists a
M\"obius transformation~$\nu\in\aut\S^3$,
an~$\cE$-decomposable biquadratic birational map~$\varphi\c\P^1\times\P^1\dto\nu(X)$
with components $(\varphi_0,\ldots,\varphi_4)$,
and complex points $p:=(1:a)$ and $q:=(1:b)$ in $\P^1$
\st
$
\varphi_0(p,q)=\cdots=\varphi_4(p,q)=0
$.
Moreover,
the left image associated to $p$ and the right image
associated to $q$ are complex lines in $\cE$ and thus not complex points.

By \Cref{lem:Q}, we may assume \Wlog that
there exist linear polynomials $S,\tS\in\H[s]$ and $T,\tT\in\H[t]$
such that the associated quaternionic polynomial~$Q$ of~$\varphi$
is equal to either
$S\cdot \tT\cdot \tS\cdot T$ or $S\cdot\tS\cdot \tT\cdot T$.

Now suppose that $Q=S\cdot \tT\cdot \tS\cdot T$.
We find that $S(a)\cdot\tT(b)\cdot\tS(a)\cdot T(b)=0$.
Moreover, the complex lines~$[S\cdot\tT(b)\cdot\tS\cdot T(b)]$
and $[S(a)\cdot\tT\cdot\tS(a)\cdot T]$ in the null quadric~$\cN$
correspond to the central projections
of the left image associated to~$p$ and the right image associated to~$q$ in $\cE$.
In particular, these central projections are not complex points.
Therefore, it follows from \RL{ST}{c} that
\Wlog $Q=S\cdot\tS\cdot \tT\cdot T$.

We established that \Wlog $Q=S\cdot\tS\cdot \tT\cdot T$
and thus the proof is concluded by \Cref{lem:SSTT}.
\end{proof}

\begin{proof}[Proof of \Cref{thm:c}]
Let $X:=\bS(Z)$ so that
$X\subset\S^3$ is a celestial surface of degree~$d$.

\ref{thm:c:a}
This assertion follows from \Cref{thm:A}\ref{thm:A:a} and thus
in particular $d\in\{2,4,8\}$.

\ref{thm:c:b}
By Assertion~\ref{thm:c:a}, $d\in\{2,8\}$.
If $d=2$, then $Z$ is M\"obius equivalent to the plane~$A+B$ for some coplanar circles $A,B\subset\R^3$
(recall that $A+B$ is Zariski closed by definition).
If $d=8$, then by \Cref{thm:A}\ref{thm:A:b} and \Cref{prp:EU}
there exists a M\"obius transformation $\nu\in\aut\S^3$
and a biquadratic birational morphism $\varphi\c \P^1\times\P^1\to \nu(X)$
that is either $\cE$-decomposable or $\cU$-decomposable.
Assertion~\ref{thm:c:b} now follows from Propositions~\ref{prp:ECD} and~\ref{prp:UAB}.

\ref{thm:c:c}
By Assertion~\ref{thm:c:a}, $d\in\{2,4\}$ and thus
Assertion~\ref{thm:c:c} now follows from \Cref{prp:E4}.

\ref{thm:c:d}
We know from Assertion~\ref{thm:c:a} and \Cref{thm:B}\ref{thm:B:a} that $d\in\{2,8\}$.
If $d=2$, then Assertion~\ref{thm:c:d} follows from \Cref{thm:B}\ref{thm:B:b}.
If $d=8$, then Assertion~\ref{thm:c:d} follows from
\Cref{thm:A}\ref{thm:A:b} and \Cref{prp:ABCD} in Appendix~\ref{app:ABCD}.
\end{proof}

\begin{lemma}
\label{lem:cubic}
The surface $Z\subset \R^3$ is a smooth cubic
that is either $2$-circled or $6$-circled
if and only if
$Z$ is M\"obius equivalent to the surface
\[
Z_\alpha:=\set{(x,y,z)\in\R^3}{f_\alpha:=z(x^2+y^2+z^2)+\alpha_1\,x^2+\alpha_2\,y^2+\alpha_3\,z=0},
\]
for some $\alpha\in \R^3$ \st
$\Psi_\alpha:=\alpha_1\alpha_2\alpha_3(\alpha_1-\alpha_2)(\alpha_1^2+\alpha_3)(\alpha_2^2+\alpha_3)\neq 0$.
\end{lemma}

\begin{proof}
For all $t\in \C^n$, we define $\sgn(t):=(\sgn(t_i))_{1\leq i\leq n}$,
where the \df{sign function} $\sgn\c \C\to \{0,1,-1\}$ is defined as
\[
\sgn(t):=
\begin{cases}
 1 & \text{if }t\in\R_{>0}, \\
-1 & \text{if }t\in\R_{<0}, \\
 0 & \text{otherwise}.
\end{cases}
\]
Let $c:=\sqrt{-\alpha_3}$
and
\[
\gamma^\alpha:=\left(
\frac{\alpha_2}{\alpha_1},~
\frac{\alpha_2+c}{\alpha_1+c},~
\frac{\alpha_2-c}{\alpha_1-c},~
\frac{\alpha_1(\alpha_1^2-c^2)}{\alpha_1-\alpha_2},~
\frac{\alpha_2(\alpha_2^2-c^2)}{\alpha_2-\alpha_1}\right).
\]
Notice that $\gamma^\alpha\in(\C\setminus\{0\})^5$, since $\Psi_\alpha\neq 0$ by assumption.
Let
\[
\lambda_\alpha:=2\,|\set{i\in I}{\sgn(\gamma_i^\alpha)=-1}|.
\]
{\bf Claim~1.} {\it Up to M\"obius equivalence of $Z_\alpha$, we have $\lambda_\alpha\in\{2,6\}$.}

Up to the M\"obius transformations $(x,y,z)\mapsto (y,x,z)$ and $(x,y,z)\mapsto (x,y,-z)$,
we may assume that $\alpha_1>\alpha_2$ and
\[
\sgn(\alpha)\in\{(1,-1,-1),(1,1,-1),(1,1,1),(1,-1,1)\}.
\]
We make a case distinction:
\begin{Mlist}
\item
If $\sgn(\alpha)=(1,-1,-1)$, then
$c\in\R_{>0}$,
$\sgn(\gamma^\alpha_1)=-1$,
$\sgn(\gamma^\alpha_2)= \sgn(\alpha_2+c)$,
$\sgn(\gamma^\alpha_3)=-\sgn(\alpha_1-c)$,
$\sgn(\gamma^\alpha_4)= \sgn(\alpha_1-c)$,
$\sgn(\gamma^\alpha_5)=-\sgn(\alpha_2+c)$,
and thus $\lambda_\alpha=6$.

\item
If $\sgn(\alpha)=(1,1,-1)$, then
$c\in\R_{>0}$,
$\sgn(\gamma^\alpha_1)=
 \sgn(\gamma^\alpha_2)=1$,
$\sgn(\gamma^\alpha_3)= \sgn(\alpha_2-c)\cdot\sgn(\alpha_1-c)$,
$\sgn(\gamma^\alpha_4)= \sgn(\alpha_1-c)$ and
$\sgn(\gamma^\alpha_5)=-\sgn(\alpha_2-c)$.
If $\sgn(\alpha_1-c)=-1$, then $\sgn(\alpha_2-c)=-1$
and thus $\lambda_\alpha=2$.

\item
If $\sgn(\alpha)=(1,1,1)$,
then
$c=\ii\,\sqrt{\alpha_3}$,
$\sgn(\gamma^\alpha_1)=1$,
$\sgn(\gamma^\alpha_2)=\sgn(\gamma^\alpha_3)=\sgn(c)=0$,
$\sgn(\gamma^\alpha_4)=\sgn(\alpha_1^2+\alpha_3)=1$,
$\sgn(\gamma^\alpha_5)=-\sgn(\alpha_2^2+\alpha_3)=-1$
and thus $\lambda_\alpha=2$.

\item
If $\sgn(\alpha)=(1,-1,1)$,
then
$c=\ii\,\sqrt{\alpha_3}$,
$\sgn(\gamma^\alpha_1)=-1$,
$\sgn(\gamma^\alpha_2)=\sgn(\gamma^\alpha_3)=\sgn(c)=0$,
$\sgn(\gamma^\alpha_4)=\sgn(\gamma^\alpha_5)=
\sgn(\alpha_1^2+\alpha_3)=\sgn(\alpha_2^2+\alpha_3)=1$
and thus $\lambda_\alpha=2$.

\end{Mlist}
This concludes the proof of Claim~1.

Let $\bC(U)\subset\C^n$ for $U\subset\R^n$ denote the Zariski closure of the
embedding of~$U$ into~$\C^n$ via the standard embedding $\R^n\hookrightarrow\C^n$.
We call $C\subset \C^n$ a \df{circle} if $C=\bC(D)$ and $D\subset\R^n$ is a circle.
Let $\mu_\C\c\bC(S^3)\dto \C^3$ denote the stereographic projection
$\mu\c S^3\dto \R^3$ extended to the complex numbers
and let $p_\infty:=(0,0,0,1)$ denote the center of stereographic projection.
Let $Y_\alpha\subset S^3$ denote the Zariski closure of $\mu^{-1}(Z_\alpha)$.
It follows from \citep[Proposition~5 and Theorem~1]{2021circ} that $\deg Y_\alpha=4$
and thus $Z_\alpha$ is a Darboux cyclide.
We remark that a Gr\"obner basis elimination shows that
\[
\bC(Y_\alpha)=\set{x\in\bC(S^3)}{\alpha_1 x_1^2 + \alpha_2 x_2^2  + \alpha_3x_3(1-x_4) + x_3(1 + x_4)=0}.
\]
Since $\Psi_\alpha\neq 0$ and
\[
f_\alpha=z(z-c)(z+c)+(z+\alpha_1)x^2+(z+\alpha_2)y^2,
\]
we find that for all $i\in I:=\{1,\ldots,5\}$,
\[
H^\alpha_i:=\set{(x,y,z)\in\C^3}{g_i=h_i=0}
\]
is a complex plane section of~$\bC(Z_\alpha)$, where
\begin{align*}
g&:=(z,~z-c,~z+c,~z+\alpha_1,~z+\alpha_2),
\\
h&:=\left(x^2+\gamma^\alpha_1\,y^2,~x^2+\gamma^\alpha_2\,y^2,~x^2+\gamma^\alpha_3\,y^2,~y^2+\gamma^\alpha_4,~y^2+\gamma^\alpha_5\right).
\end{align*}
We observe that $\Psi_\alpha\neq 0$
if and only if
the union~$\cup_{i\in I} H^\alpha_i$ of the 5 complex plane sections consists of 10 distinct complex lines.
Indeed, $\Psi_\alpha\neq 0$ implies for all $i\in I$ that $\gamma^\alpha_i\in\C\setminus\{0\}$
and thus the polynomial~$h_i$ has two distinct linear factors.
Let $\ell_\alpha(q)$ and $\ell_\alpha^\R(q)$ denote the number of complex and real conics
in~$\bC(Y_\alpha)$ containing $q\in \bC(Y_\alpha)$, \resp.

{\bf Claim 2.} {\it
If $q\in\bC(Y_\alpha)$ is a general point, then
$\ell_\alpha(q)=\ell_\alpha(p_\infty)$,
$\ell_\alpha^\R(q)=\ell_\alpha^\R(p_\infty)$
and $Z_\alpha$ is $\ell_\alpha^\R(p_\infty)$-circled.
}

Let $C\subset \bC(Y_\alpha)$ be a complex conic containing $p_\infty$.
There exists a one-parameter family of complex hyperplane sections of $\bC(S^3)$ containing $C$.
Let $C'\subset\bC(Y_\alpha)$ \st
$C\cup C'$ is the hyperplane section of $\bC(Y_\alpha)$ containing $C\cup\{q\}$.
By B\'ezout's Theorem, $C'$ is a complex conic containing $q$.
We observe that $C$ is real if and only if $C'$ is real.
Since $q$ is general and $\bC(Y_\alpha)$ is not covered by complex lines,
$C'$ must be irreducible. Moreover, since the real point set $Y_\alpha$ is a surface,
we deduce that if $C'$ is real, then $C'_\R\neq \varnothing$ and thus $C'$ must be a circle.
Therefore, $Z_\alpha$ is $\ell_\alpha^\R(q)$-circled by definition,
concluding the proof of Claim~2.

{\bf Claim 3.} {\it
If $L\subset H^\alpha_i$ is a (complex) line for some $i\in I$,
then there exists a unique (complex) irreducible conic $C\subset\bC(Y_\alpha)$
\st $\mu_\C(C)=L$.}

First notice that for all $i\in I$
there exists $\alpha\in\R^3$ \st $\sgn(\gamma^\alpha_i)=-1$.
In this case, the lines in $H^\alpha_i$ are real
and thus must be projections of two cospherical circles in~$\bC(Y_\alpha)$
containing the center~$p_\infty$.
But this concludes the proof of Claim~3 as the same must hold over the complex numbers
(and thus for all $\alpha\in\R^3$).

{\bf Claim 4.} {\it
We have $\Sing \bC(Y_\alpha)=\varnothing$ and
if $C\subset\bC(Y_\alpha)$ is a (complex) irreducible conic \st $p_\infty\in C$,
then there exists a unique $i\in I$ \st $\mu_\C(C)\subset H^\alpha_i$
is a (complex) line.}

Suppose that $q\in\bC(Y_\alpha)$ is a general point
so that $\ell_\alpha(q)$ is the number of complex irreducible conics containing $q$.
It follows from \citep[Corollary~6]{2021circ} that $\ell_\alpha(q)\leq 10$.
Moreover, if $\ell_\alpha(q)=10$, then $\Sing\bC(Y_\alpha)=\varnothing$
(in \citep[Table~6]{2021circ} we have $\ell_\alpha(q)=|G(X)|$ and $|\Sing\bC(Y_\alpha)|\leq |B(X)|$
as a straightforward consequence of the definitions; see \citep[\textsection 4]{2023trans} for more details
or \citep[\textsection2.4]{2012web} for an alternative approach).
We know from Claim~2 that $\ell_\alpha(q)=\ell_\alpha(p_\infty)$.
Since $\Psi_\alpha\neq 0$ the union~$\cup_{i\in I} H^\alpha_i$ consists of 10 distinct complex lines.
By Claim~3 each of the 10 complex irreducible conics in~$\bC(Y_\alpha)$ containing~$p_\infty$ are accounted for
and thus Claim~4 holds true.

We are now ready to prove the $\Leftarrow$ direction.
By Claims~3 and~4, the real lines in~$\cup_{i\in I} H^\alpha_i$
are in bijective correspondence with the $k$ circles in $Y_\alpha$ containing $p_\infty$ for some $k\in\Z_{\geq 0}$.
A pair of complex lines in $H^\alpha_i$ is real if and only if $\sgn(\gamma^\alpha_i)=-1$.
Thus, we established that $k=\lambda_\alpha=\ell_\alpha^\R(p_\infty)$.
Hence, $Z_\alpha$ is $\lambda_\alpha$-circled by Claim~2
and we know from Claim~1 that $\lambda_\alpha\in\{2,6\}$.
Since $\Sing\bC(Y_\alpha)=\varnothing$ by Claim~4, we deduce that $\Sing Z_\alpha=\varnothing$.
This concludes the proof for the $\Leftarrow$ direction.

Let us proceed with proving the $\Rightarrow$ direction.
We know from \citep[Theorem~6.6 with Propositions~3.2 and~4.2]{2000}
that there exists $\beta\in\R^3$ \st the $\lambda$-circled cubic surface~$Z$ is M\"obius equivalent to the surface
\[
X_\beta:=\set{(x,y,z)\in\R^3}{(x^2+y^2+z^2)^2+\beta_1\,x^2+\beta_2\,y^2+\beta_3\,z^2+\lambda-4=0}.
\]
In particular, this shows that
for all $\alpha\in\R^3$
there exists $\beta\in\R^3$ \st $Z_\alpha$ is M\"obius equivalent
to $X_\beta$
(notice that in \cite{2000} we have \Wlog $a^2=2$;
see \citep[Remark~32]{2023trans} for an alternative proof strategy).
It remains to show that for all $\beta\in\R^3$
the surface $X_\beta$ is M\"obius equivalent to $Z_\alpha$ for some $\alpha\in\R^3$.
If $\sgn(\gamma^\alpha_i)=-1$ for some $i\in I$,
then $H_i\subset Z_\alpha$ consists of two real lines
that intersect along an angle that continuously depends on $\gamma^\alpha_i$.
Since the M\"obius transformations are angle preserving,
it follows that $\gamma^\alpha_i$ is a M\"obius invariant.
This extends to the complex numbers and thus $\gamma^\alpha\in\C^5$ is a M\"obius invariant for all $\alpha\in\R^3$.
It is straightforward to see that $\set{(\gamma^\alpha_1,\gamma^\alpha_2,\gamma^\alpha_4)}{\alpha\in \R^3}$
is Zariski dense in~$\C^3$.
This implies that the set of M\"obius equivalence classes~$\set{[Z_\alpha]}{\alpha\in \R^3}$
is 3-dimensional.
It follows that for all $\beta\in\R^3$ there exists $\nu\in \aut\S^3$ and $\alpha\in\R^3$ \st
\[
\nu(\bS(X_\beta))=\bS(Z_\alpha).
\]
This concludes the proof as $Z$ is M\"obius equivalent to $Z_\alpha$ for some $\alpha\in\R^3$.
\end{proof}

\begin{proof}[Proof of \Cref{cor:c}]
Let $X:=\bS(Z)$ so that $X$ is a $\lambda$-circled celestial surface of degree~$d$.
Let $p_\infty:=(0,0,0,1)$ denote the center of stereographic projection $\mu\c S^3\dto \R^3$.
Notice that $Z$ is the stereographic projection of~$\bR(X)$
and the degree of $Z$ is equal to~$d-m$, where $m\geq 0$ is the multiplicity of $p_\infty$ in $\bR(X)$.
We know from \Cref{thm:c}\ref{thm:c:a} that $d\in\{2,4,8\}$.
We make a case distinction on $d$.

Suppose that $d=2$.
In this case, $Z$ is of degree at most two and M\"obius equivalent to an $\infty$-circled plane.
By Theorems~\ref{thm:c}\ref{thm:c:b} and~\ref{thm:c}\ref{thm:c:c}, $Z$ belongs to~\ref{c:ii} only.

Suppose that $d=8$.
In this case, $\deg Z\geq 4$ by \citep[Proposition~5]{2021circ}
and thus $Z$ does not belong to~\ref{c:iii} or~\ref{c:iv}.
By \citep[Corollary~2]{2023trans} the sum of two circles is either
$\infty$-circled or 2-circled.
Hence, $Z$ belongs to exactly one of~\ref{c:i} or~\ref{c:ii}
by Theorems~\ref{thm:c}\ref{thm:c:b} and~\ref{thm:c}\ref{thm:c:d}.

Suppose that $d=4$ and $\bR(X)$ is smooth.
We may assume up to M\"obius equivalence that $p_\infty\in \bR(X)$ and thus
$Z$ is a smooth cubic surface and a celestial Darboux cyclide.
If $\lambda\in\{2,6\}$, then
$Z$ belongs to~\ref{c:iii} by \Cref{lem:cubic}.
It follows from Theorems~\ref{thm:B}\ref{thm:B:a} and~\ref{thm:B}\ref{thm:B:d}  that
$\lambda\in \{2,4,5,6\}$ and $Z$ does not belong to~\ref{c:ii}.
If $\lambda\in\{4,5\}$, then $Z$ belongs only to~\ref{c:i} by \Cref{thm:c}\ref{thm:c:c}.
Thus, if $\lambda\in\{2,6\}$, then $Z$ belongs only to~\ref{c:iii}.

Suppose that $d=4$ and $\bR(X)$ is singular.
We may assume up to M\"obius equivalence that $p_\infty\in \Sing \bR(X)$ so that $Z$ is a quadratic surface.
The M\"obius transformations preserving the point $\bP(p_\infty)\in\S^3$
correspond via the stereographic projection to Euclidean similarities of $\R^3$.
The classification of quadratic surfaces in $\R^3$ up to Euclidean similarities is classically known.
Indeed, it follows from \citep[Proposition~4]{2021circ} that
the projective model $\bP(Z)\subset\P^3$ is up to Euclidean similarities defined
by one of the 24 rows in \citep[Table~11]{2021circ}.
We refer to each of these rows by the name specified in the last column.
See \citep[Definition~3]{2021circ} for an explanation of these names
and note in particular that ``CH1'' is short for ``one-sheeted circular hyperboloid''.
The first column of these rows shows
that there exists $a\in\R_{>0}$, $b\in\R_{>0}\cup\{-1\}$ and $c,d,e\in\{-1,0,1\}$ \st
\[
\bP(Z)=\set{x\in\P^3}{a\,x_1^2+b\,x_2^2+c\,x_3^2+d\,x_0\,x_3+e\,x_0^2=0}.
\]
The third and fourth columns specify the values $C,L\in\Z_{\geq 0}$
\st $Z$ contains $C$ circles and $L$ lines through almost each point.
Since $Z$ is $(C+L)$-circled, we have $C+L\geq 2$ by assumption.
Moreover, $X$ is singular and thus $\bP(Z)$ does not belong to the row with name~$\S^2$.
It follows that $\bP(Z)$ is defined by exactly one of the following 10 rows:
EE, EH2, EH1, CH1, EO, CO, EP, EY, CY or HP.
We make a case distinction on the three inequalities at~\ref{c:iv}.
\begin{Mlist}
\item Firstly, suppose that $(a-b)(b+1)c(d^2-1)e(c+e-2)\neq 0$.
The factor $(d^2-1)$ implies that $d=0$.
The factor $c\,e\,(c+e-2)$ implies that $c=\pm 1$, $e=\pm 1$ and $(c,e)\neq (1,1)$.
The factor $(a-b)(b+1)$ implies that $a,b>0$ and $a\neq b$.
We find that $\bP(Z)$ is defined by row either EE, EH2 or EH1.

\item Secondly, suppose that $(b+1)(c-1)(d^2-1)(e-1)(c^2e^2-1)\neq 0$.
The factor $(d^2-1)$ implies that $d=0$.
The factor $(c-1)(e-1)(c^2e^2-1)$ implies that $c,e\in\{0,-1\}$ and $(c,e)\neq(-1,-1)$.
The factor $(b+1)$ implies that $a,b>0$.
We find that $\bP(Z)$ is defined by row either EO, CO, EY or CY.

\item Lastly, suppose that
$(a-b)(c^2-1)(d^2+d)(e^2-1)\neq 0$.
The factor $(d^2+d)$ implies that $d=1$.
The factor $(c^2-1)(e^2-1)$ implies that $c=e=0$.
The factor $(a-b)$ implies that $a>0$, $a\neq b$ and either $b>0$ or $b=-1$.
We find that $\bP(Z)$ is defined by row either EP or HP.
\end{Mlist}
We established that $Z$ belongs to~\ref{c:iv}
if and only if $\bP(Z)$ is not CH1.
By Theorems~\ref{thm:c}\ref{thm:c:c} and~\ref{thm:c}\ref{thm:c:d},
$Z$ belongs to~\ref{c:i} if and only if $\bP(Z)$ is CH1.

We considered all possible disjoint cases and thus concluded the proof.
\end{proof}

\appendix
\addappheadtotoc

\section{The proof of Theorem B}
\label{app:B}

\Cref{thm:B} follows from the classifications in~\cite{2023trans}.
However, for \Cref{thm:B}\ref{thm:B:e}, we need to translate some results from
\cite{2023trans} to our current setting.

\begin{proof}[Proof of \Cref{thm:B}]
Assertions~\ref{thm:B:a}, \ref{thm:B:b} and \ref{thm:B:c}
follow from \citep[Theorem~1(b) and Corollary~2]{2023trans}.
For Assertion~\ref{thm:B:a},
notice that the circles $A$ and $B$ must be coplanar.
Assertion~\ref{thm:B:d} follows from \citep[Theorem~A and Table~5]{2023trans}.

\ref{thm:B:e}
We first recall some terminology
from \cite{2023trans}.

By \citep[Theorem~2.16]{2007}, there exists up to biregular isomorphisms
a unique birational morphism~$f\c Y\to X$
such that $Y$ is a smooth surface,
$f$ does not contract $(-1)$-curves and $f\circ \sigma_Y=\sigma_X\circ f$.

The \df{N\'eron-Severi lattice} $N(X)$ is
an additive group  defined by the divisor classes on~$Y$ up to numerical equivalence.
This group comes with a bilinear symmetric intersection product $\_\cdot\_\c N(X)\times N(X)\to\Z$
and an involution
$
\sigma_{X*}\c N(X)\to N(X)
$
induced by $\sigma_Y$.
We denote by $\aut N(X)$ the group automorphisms that are compatible
with both $\cdot$ and~$\sigma_{X*}$.

We remark that on rational surfaces numerical and linear equivalence coincide
(see the proof of~\citep[Lemma~4.9]{2018sko})
and that $N(X)$ does not depend on the choice of $f$. In particular, we have $N(X)=N(Y)$.

The \df{class}~$[C]$ of a complex curve $C\subset X$ is defined as the divisor class of $C_Y$ in $N(X)$,
where $C_Y\subset Y$ is the union of complex curves
in~$f^{-1}(C)$ that are not contracted to complex points
by the morphism~$f$. Beware that the square bracket notation is also
used for a different concept (see \Cref{sec:clifford}).

Since $X\subset\S^3$ is by assumption either a Perseus cyclide or CH1 cyclide,
it follows from \citep[Theorem~A]{2023trans} (and thus \citep[Theorem~4]{2021circ}) that
\[
N(X)\cong\bas{\l_0,\l_1,\p_1,\p_2,\p_3,\p_4}_\Z,
\]
$\sigma_{X*}(\l_0)=\l_0$,
$\sigma_{X*}(\l_1)=\l_1$,
$\sigma_{X*}(\p_1)=\p_2$,
$\sigma_{X*}(\p_3)=\p_4$,
and the non-zero intersections between the generators are $\l_0\cdot\l_1=1$ and $\p_1^2=\p_2^2=\p_3^2=\p_4^2=-1$.
The class of a hyperplane section of $X$ is
equal to $-\k_Y$, where $\k_Y:=-2\,\l_0-2\,\l_1+\p_1+\cdots+\p_4$ denotes the \df{canonical class} of $Y$
(see \citep[Section~2]{2021circ}).

We consider the following subsets of $N(X)$:
\begin{Mlist}
\item $B(X)$ denotes the set of divisor classes of complex irreducible curves $C\subset Y$
\st $f(C)$ is a complex point in~$X$,
\item $G(X)$ denotes the set of classes of complex irreducible conics in $X$, and
\item $E(X)$ denotes the set of classes of complex lines in $X$.
\end{Mlist}

We use the following shorthand notation for elements in $N(X)$:
\begin{gather*}
\begin{array}{@{}l@{\hspace{1cm}}l@{\hspace{1cm}}l@{}}
b_1:=\p_1-\p_3, & b_{ij}:=\l_0-\p_i-\p_j   & b_0:=\l_0+\l_1-\p_1-\p_2-\p_3-\p_4, \\
b_2:=\p_2-\p_4, & b_{ij}':=\l_1-\p_i-\p_j, &
e_i:=\p_i,\quad e_{ij}:=\l_i-\p_j,\quad e_i':=b_0+\p_i.
\end{array}
\end{gather*}

{\bf Claim~1.}
{\it
If, up to $\aut N(X)$, we have
$\{b_1,b_2\}\subseteq B(X)\subseteq\{b_1,b_2,b_0\}$,
\[
\{\l_0,\l_1\}\subset G(X)
\quad\text{and}\quad
\{e_3,e_4,e_{01},e_{02},e_{11},e_{12}\}\subset E(X),
\]
then there exist a birational morphism~$g\c Y\to\P^1\times\P^1$
and complex conjugate curves
$
A,\oA,C,\oC,D,\oD\subset Y
$
\st
\begin{Mlist}
\item $f(A)$, $f(\oA)$, $f(C)$, $f(\oC)$, $f(D)$, $f(\oD)$
are complex lines in $X$ with classes
$e_3$, $e_4$, $e_{01}$, $e_{02}$, $e_{11}$, $e_{12}$, \resp,
\item $f\circ g^{-1}\c \P^1\times\P^1\dto X$ is a birational map of bidegree (2,2)
that is not defined at $(p,q):=g(A)$ and $(\op,\oq):=g(\oA)$, and
\item
$g(C)$ and $g(\oC)$
are the left fibers associated to $p$ and $\op$, \resp, and
\item
$g(D)$ and $g(\oD)$ are
the right fibers associated to $q$ and $\oq$, \resp.
\end{Mlist}
}

We let the complex curves~$A$, $\oA$, $C$, $\oC$, $D$ and $\oD$ in $Y$ be the strict transforms via~$f$
of the complex lines in~$X$ with classes
$e_3$, $e_4$, $e_{01}$, $e_{02}$, $e_{11}$ and $e_{12}$, \resp.
It follows from the arithmetic and geometric genus formulas \citep[Exercises~IV.1.8a and V.1.3]{1977}
that $p_a(A)=p_g(A)=p_a(\oA)=p_g(\oA)=0$ and thus $A\cong\oA\cong\P^1$.
Since $[A]^2=[\oA]^2=-1$,
Castelnuovo's contraction theorem implies
that there exists a birational morphism $u\c Y\to Y'$ that contracts
$A$ and $\oA$ to complex conjugate points~$P_3$ and~$P_4$ on a smooth surface~$Y'$
(see \citep[Theorem~V.5.7]{1977}).
We know from \citep[Propositions~V.3.2 and~V.3.3]{1977} that
$N(X)=N(Y')\oplus\bas{\p_3,\p_4}_\Z$ and $\k_{Y'}=-2\,\l_0-2\,\l_1+\p_1+\p_2$
so that $N(Y')\cong\bas{\l_0,\l_1,\p_1,\p_2}_\Z$.
We deduce from the Riemann-Roch theorem \citep[Theorem~V.1.6]{1977}
that there exist complex conjugate
curves~$B,\oB\subset Y'$ \st $[B]=\p_1$ and $[\oB]=\p_2$.
As before, the genus formulas
and Castelnuovo's contraction theorem imply that there exists a birational morphism
$v\c Y'\to Y''$ that contracts $B$ and $\oB$ to complex conjugate points $P_1$ and $P_2$
on a smooth surface $Y''$
with $N(Y'')=\bas{\l_0,\l_1}_\Z$ and $\k_{Y''}=-2\,\l_0-2\,\l_1$.
We let $g:=v\circ u$ be the composition of blowups $u$ and $v$.
Since, $\k_{Y''}^2=8$, $-\k_{Y''}\cdot \l_i>0$ and $g^*(\l_i)\in G(X)$ for $i\in\{0,1\}$,
it follows that $-\k_{Y''}$ is ample
and thus $Y''$ is a del Pezzo surface of degree 8
so that $Y''\cong\P^1\times\P^1$ (see \citep[Definition~8.1.12 and \textsection8.4.1]{2012dol}).
As $b_1,b_2\in B(X)$, we deduce that $P_3\in B$ and $P_4\in\oB$
and thus $P_3$ and $P_4$ are infinitely near to~$P_1$ and $P_2$,
\resp~(see \citep[Definition at page~392]{1977}).
We set $(p,q):=P_1$ and $(\op,\oq):=P_2$.
Since $-\k_Y$ is the class of hyperplane section of $X$ and $g_*(\k_Y)=\k_{Y''}$,
it follows that the components of the map $f\circ g^{-1}$
form a basis of the 5-dimensional vector space of bidegree (2,2) forms that
vanish at $P_1$, $P_2$ and the infinitely near points $P_3$, $P_4$.
Thus the complex tangent directions at $P_1$ and $P_2$
of the zero sets of these forms are determined by $P_3$ and $P_4$, \resp.
We recover the 5-dimensional subspace of the 9-dimensional vector space of bidegree~(2,2) forms
by assigning the four conditions corresponding to $P_1$, $P_2$, $P_3$ and $P_4$.
The map $f$ associated to~$-\k_Y$ is birational by \citep[Theorem~8.3.2]{2012dol}
and thus $f\circ g^{-1}$ is birational as well.
Since $u(A)\in B$ and $u(\oA)\in\oB$, we established that $g(A)=P_1=(p,q)$ and $g(\oA)=P_2=(\op,\oq)$.
Notice that the classes $\l_0$ and $\l_1$ in $N(\P^1\times\P^1)$
are the classes of a left fiber and right fiber, \resp.
Therefore,
the preimage via $g$ of the left fiber associated to $p$ has class $e_{01}$ in $Y$ by \citep[Proposition~V.3.6]{1977}.
Similarly,
the preimage of the right fiber associated to $q$ has class $e_{11}$.
From this we deduce that
$g(C)$, $g(\oC)$, $g(D)$ and $g(\oD)$ are as asserted
and we concluded the proof of~Claim~1.

If $X$ is a CH1 cyclide, then
it follows from \citep[Theorem~A]{2023trans} that, up to $\aut N(X)$,
\[
B(X)=\{b_{13},b_{24},b_{12}'\},~
\{g_0,g_1,g_{34}\}\subset G(X),~
E(X)=\{e_1,e_2,e_3,e_4,e_{13},e_{14}\},
\]
where $g_0:=\l_0$, $g_1:=\l_1$ and $g_{34}:=\l_0+\l_1-\p_1-\p_2$.
Let $\gamma\in\aut N(X)$ \st
\begin{gather*}
\gamma(\l_0)=\l_0,\quad
\gamma(\l_1)=\l_0+\l_1-\p_1-\p_2,\quad
\\
\gamma(\p_1)=\p_3,\quad
\gamma(\p_2)=\p_4,\quad
\gamma(\p_3)=\l_0-\p_1,\quad
\gamma(\p_4)=\l_0-\p_2.
\end{gather*}
After applying $\gamma$, we find that if $X$ is a CH1 cyclide, then, up to $\aut X$,
\[
B(X)=\{b_1,b_2,b_0\},~~
\{\l_0,\l_1\}\subset G(X),~~
E(X)=\{e_3,e_4,e_{01},e_{02},e_{12},e_{11}\}.
\]
If $X$ is a Perseus cyclide, then it follows from \citep[Theorem~A]{2023trans} that, up to $\aut N(X)$,
\[
B(X)=\{b_1,b_2\},~~
\{\l_0,\l_1\}\subset G(X),~~
E(X)=\{e_3,e_4,e_{01},e_{02},e_{11},e_{12},e_4',e_3'\}.
\]
Suppose that $X\subset\S^3$ is a Perseus cyclide or CH1 cyclide.
We apply Claim~1 and find that $\varphi:=\nu\circ f\circ g^{-1}$
is a biquadratic and birational map that is not defined at
the complex conjugate points $(p,q)$ and $(\op,\oq)$ in~$\P^1\times\P^1$
for all $\nu\in\aut\S^3$.
It follows from \citep[Proposition~13 (see also Example~15)]{2023trans} that
the incidences between complex lines in $X$ are as in \Cref{fig:lines}.
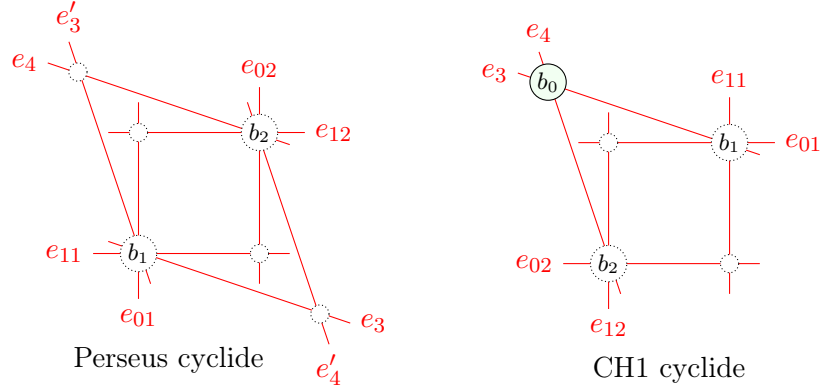
\begin{figure}[!ht]
\centering
\setlength{\tabcolsep}{5mm}
\begin{tabular}{cc}
\begin{tikzpicture}[scale=0.4] 
\node at (0,6) {};\node at (0,-6) {};
\draw[red] (-3,2) -- (3.5,2) node[right] {$e_{12}$};
\draw[red] (-3.5,-2) node[left] {$e_{11}$} -- (3,-2);
\draw[red] (-2,3) -- (-2,-3.5) node[below] {$e_{01}$};
\draw[red] (2,3.5) node[above] {$e_{02}$} -- (2,-3);
\draw[red] (-5,4.3) node[left,red]  {$e_4$} -- (3,1.6);
\draw[red] (-4.3,5) node[above,red] {$e_3'$} -- (-1.6,-3);
\draw[red] (5,-4.3) node[right,red] {$e_3$} -- (-3,-1.6);
\draw[red] (4.3,-5) node[below,red] {$e_4'$} -- (1.6,3);
\draw[draw=black, fill=white, densely dotted] (-2,2) circle [radius=3mm];
\draw[draw=black, fill=white, densely dotted] (2,-2) circle [radius=3mm];
\draw[draw=black, fill=white, densely dotted] (-4,4) circle [radius=3mm];
\draw[draw=black, fill=white, densely dotted] (4,-4) circle [radius=3mm];
\draw[draw=black, fill=white, densely dotted] (-2,-2) circle [radius=6mm] node[black] {\footnotesize $b_1$};
\draw[draw=black, fill=white, densely dotted] (2,2)   circle [radius=6mm] node[black] {\footnotesize $b_2$};
\node at (-1,-5.5) {Perseus cyclide};
\end{tikzpicture}
&
\begin{tikzpicture}[scale=0.4] 
\node at (0,6) {};\node at (0,-6) {};
\draw[red] (-3,2)  -- (3.5,2) node[right] {$e_{01}$};
\draw[red] (-3.5,-2) node[left] {$e_{02}$} -- (3,-2);
\draw[red] (-2,3) -- (-2,-3.5) node[below] {$e_{12}$};
\draw[red] (2,3.5) node[above] {$e_{11}$} -- (2,-3);
\draw[red] (-5,4.3) node[left, red] {$e_3$} -- (3,1.6);
\draw[red] (-4.3,5) node[above, red] {$e_4$} -- (-1.6,-3);
\draw[draw=black, fill=white, densely dotted] (-2,2) circle [radius=3mm];
\draw[draw=black, fill=white, densely dotted] (2,-2) circle [radius=3mm];
\draw[draw=black, fill=white, densely dotted] (2,2) circle [radius=6mm] node[black] {\footnotesize$b_1$};
\draw[draw=black, fill=white, densely dotted] (-2,-2) circle [radius=6mm] node[black] {\footnotesize$b_2$};
\draw[draw=black, fill=green!5] (-4,4) circle [radius=6mm] node[black] {\footnotesize$b_0$};
\node at (0,-5.5) {CH1 cyclide};
\end{tikzpicture}
\end{tabular}
\caption{
Incidences between complex lines and isolated singularities in a
Perseus cyclide and CH1 cyclide~$X\subset\S^3$.
Each complex line is represented as a line segment and labeled with its corresponding class in $E(X)$.
If $[C]\in B(X)$, then the complex point $f(C)$
is represented as a disc labeled with $[C]$.
If $f(C)$ is real or non-real, then the disc has a solid and dashed
border, \resp.
}
\label{fig:lines}
\end{figure}
Hence, we conclude from Claim~1 and \Cref{lem:S3} that there exists
a M\"obius transformation $\nu\in\aut\S^3$ \st
the two pairs
\[
L_p:=\nu\circ f(C),\quad
L_\op:=\nu\circ f(\oC)
\quad\text{and}\quad
R_q:=\nu\circ f(D),\quad
R_\oq:=\nu\circ f(\oD)
\]
of skew complex conjugate lines are contained in~$\cE$.
Notice that $L_p$ is the left image associated to~$p$,
since $\varphi=\nu\circ f\circ g^{-1}$ and
$g(C)$ is by Claim~1 the left fiber associated to~$p$.
Similarly, for $L_\op$, $R_q$ and $R_\oq$.
In particular, this shows that $\varphi$ is $\cE$-decomposable
and thus we concluded the proof for Assertion~\ref{thm:B:e} of~\Cref{thm:B}.
\end{proof}

\section{Sums and products of circles are not equivalent}
\label{app:ABCD}

This goal of this appendix is to prove \Cref{prp:ABCD},
which states that if $X=\bS(A+B)$ for some circles~$A,B\in\R^3$,
then $\nu(X)\neq\bP(C\cdot D)$ for all circles~$C,D\subset S^3$ and M\"obius transformations~$\nu\in\aut\S^3$.

\begin{lemma}
\label{lem:LLRR}
If $\varphi\c\P^1\times\P^1\to X\subset\S^3$ is a biquadratic birational morphism
and $H\cap X$ contains four complex lines for some~$H\in\{\cE,\cU\}$,
then $\varphi$ is $H$-decomposable
and $H\cap X\subset\Sing X$ consists of four complex double lines.
\end{lemma}

\begin{proof}
Any complex line contained in $H$ is non-real
and thus the complex lines in the real hyperplane section~$H\cap X$
come in complex conjugate pairs $L$, $\oL$ and $R$, $\oR$.
We denote their preimages in $\P^1\times\P^1$ by $A$, $\oA$ and $B$, $\oB$, \resp.
We know from \RL{Y}{C} that $\deg X=8$.
Complex conjugate lines in~$X$ have equal multiplicities.
Therefore, B\'ezout's theorem implies that a complex line in $H\cap X$ has multiplicity
at most three.
Thus, for all $F\in\{A,B\}$, the general fiber of $\varphi|_F$ has cardinality $d\leq 3$.
We apply \RL{Y}{E} and find that $d=2$ and $F$ has bidegree either~$(1,0)$ or~$(0,1)$
so that it is a left fiber and right fiber, \resp.
This implies in particular that the complex lines in $X\cap H$ have multiplicity at least~$d=2$.
It follows from B\'ezout's theorem that
$X\cap H=L\cup\oL\cup R\cup\oR$ and $L,\oL,R,\oR\subset X$ are complex double lines.
The complex conjugate of a left or right fiber is again left and right, \resp.
We conclude that $\varphi$ is $H$-decomposable.
\end{proof}

We may identify $\S^3_\R\cong S^3$ with the unit-quaternions
and thus the Hamiltonian product $\_\cdot\_\c S^3\times S^3\to S^3$
extends to the rational map $\_\circledast\_\c\S^3\times\S^3\dto \S^3$ that is defined as
\begin{align*}
(x,y)\mapsto (x_0y_0:& ~x_1y_1-x_2y_2-x_3y_3-x_4y_4:x_1y_2+x_2y_1+x_3y_4-x_4y_3:\\
                     & ~x_1y_3-x_2y_4+x_3y_1+x_4y_2:x_1y_4+x_2y_3-x_3y_2+x_4y_1).
\end{align*}
Let the rational map $\_\odot\_\c\P^3\times\P^3\dto \P^3$ be defined as
\begin{align*}
(x,y)\mapsto(&x_1y_1-x_2y_2-x_3y_3-x_4y_4:x_1y_2+x_2y_1+x_3y_4-x_4y_3:\\
             &x_1y_3-x_2y_4+x_3y_1+x_4y_2:x_1y_4+x_2y_3-x_3y_2+x_4y_1),
\end{align*}
where $x=(x_1:\ldots:x_4)$ and $y=(y_1:\ldots:y_4)$.

\begin{proposition}
\label{prp:el}
We consider the diagram at \Cref{tab:el},
where $S^3$ is identified with the unit-quaternions,
the morphism $\tau$ is the 2:1 central projection
and
\[
[h_1+h_2\,\qi+h_3\,\qj+h_4\,\qk]=(h_1:h_2:h_3:h_4).
\]
\begin{table}[!ht]
\caption{See \Cref{prp:el}.}
\label{tab:el}
\centering
\begin{tikzpicture}[node distance=8mm and 20mm, auto]
\node (A) {$S^3$};
\node (B) [right= of A] {$\S^3$};
\node (C) [right= of B] {$\P^3$};
\node (D) [right= of C] {$\HC$};
\node (a) [below= of A] {$S^3$};
\node (b) [right= of a] {$\S^3$};
\node (c) [right= of b] {$\P^3$};
\node (d) [right= of c] {$\HC$};

\draw[arrows={Hooks[right]->}] (A) to node {$\bP$} (B);
\draw[->] (B) to node {$\tau$} (C);
\draw[<-] (C) to node {$[\_]$} (D);

\draw[arrows={Hooks[right]->}] (a) to node[swap] {$\bP$} (b);
\draw[->] (b) to node[swap] {$\tau$}(c);
\draw[<-] (c) to node[swap] {$[\_]$} (d);

\draw[->] (A) to node[swap] {$\psi_1$} (a);
\draw[->] (B) to node[swap] {$\psi_2$} (b);
\draw[->] (C) to node       {$\psi_3$} (c);
\draw[->] (D) to node       {$\psi_4$} (d);
\end{tikzpicture}
\end{table}

Suppose that $q\in S^3$. There exists $h\in\HC$ \st $\tau(\bP(q))=[h]$
and if we define either
\begin{Mlist}
\item
$\psi_1(x):=q\cdot x$,\quad
$\psi_2(x):=\bP(q)\circledast x$,\quad
$\psi_3(x):=[h]\odot x$,\quad
$\psi_4(x):=h\cdot x$, or
\item
$\psi_1(x):=x\cdot q$,\quad
$\psi_2(x):=x\circledast\bP(q)$,\quad
$\psi_3(x):=x\odot[h]$,\quad
$\psi_4(x):=x\cdot h$,
\end{Mlist}
then the diagram in \Cref{tab:el} commutes,
$\psi_1$ is an isometry,
$\psi_2\in\aut\S^3$ \st $\psi_2(\cE)=\cE$,
$\psi_3\in\aut\P^3$ \st $\psi_3(\cN)=\cN$, and
$\psi_4$ is bijective.
\end{proposition}

\begin{proof}
The existence of $h$ follows from the surjectivity of $[\_]$.
If $\psi_1(x)=q\cdot x$, then it follows from \RL{basic}{a} that
\[
\psi_1(x-y)^*\cdot\psi_1(x-y)=(x-y)^*\cdot q^*\cdot q\cdot(x-y)=(x-y)^*\cdot(x-y).
\]
Since the squared Euclidean distance between $x,y\in S^3$ is equals to $(x-y)^*\cdot(x-y)$,
we conclude that $\psi_1$ is an isometry.
Similarly, if $\psi_1(x)=x\cdot q$, then $\psi_1$ is again an isometry.
The commutativity of \Cref{tab:el} and the assertions about $\psi_i$ for $i\in\{2,3,4\}$
are now a direct consequence of the definitions of
the Hamiltonian product and the maps~$\psi_i$.
\end{proof}

\begin{proposition}
\label{prp:CDE}
If $\varphi\c\P^1\times\P^1\to X$ is a
biquadratic birational morphism and $X=\bP(C\cdot D)$
for some circles $C,D\subset S^3$,
then $\varphi$ is $\cE$-decomposable.
\end{proposition}

\begin{proof}
Let $A:=\bP(\{a\}\cdot D)$ and $B:=\bP(\{b\}\cdot D)$ for general $a,b\in C$.
Since $A$ is the Zariski closure of~$\set{\bP(a)\circledast\bP(d)}{d\in D}$,
it follows from \Cref{prp:el} that $A$ is a M\"obius transformation of the circle~$\bP(D)$.
Notice that $\cE\cap X$ is a hyperplane section of $X$
so that $|A\cap\cE|=2$ by B\'ezout's theorem.
Therefore,
there exist complex conjugate quaternions~$\alpha,\oalpha\in\HC$
\st the central projection~$\tau(A\cap\cE)$ equals $\{[\alpha],[\oalpha]\}\subset\cN$.
Similarly, we find that $\tau(B\cap\cE)=\{[\beta],[\obeta]\}\subset\cN$.
We know from Lemmas~\ref{lem:C}\ref{lem:C:a} and~\ref{lem:C}\ref{lem:C:b} that $|A\cap B|<2$,
which implies that $\{[\alpha],[\oalpha]\}\neq\{[\beta],[\obeta]\}$.
Let $q:=b\cdot a^*$ in $S^3$ so that $q\cdot a=b$
and let $h\in\HC$ \st $[h]=\tau(\bP(q))$.
It follows from \Cref{prp:el} with $\psi_1(x)=q\cdot x$
that $\psi_2(A)=B$, $\psi_3(\tau(A))=\tau(B)$ and $\psi_4(\alpha)=\beta$.
This implies that $[h\cdot\alpha]=[\beta]$
and thus $[\alpha],[\beta]\in\cL_\alpha$ by \RL{N}{d}.
Since $a,b\in C$ are general and $|A\cap B\cap\cE|=0$,
we deduce that $X\cap\cE$ contains the complex conjugate left generators
$\cL_\alpha$ and $\cL_\oalpha$, which are complex lines by \RL{N}{a}.
The analoguous arguments for instead $C\cdot\{a\}$ and $C\cdot\{b\}$
show that
there exists $\alpha\in\HC$ \st
$\cR_\alpha$ and $\cR_\oalpha$
are complex conjugate lines in $X\cap\cE$.
Thus $X\cap\cE$ consists of four complex lines
so that $\varphi$ is $\cE$-decomposable by \Cref{lem:LLRR}.
\end{proof}

The following lemma is the converse of \Cref{lem:LLpar}.
Recall that two circles in $\R^3$ are parallel if their spanning planes are.

\begin{lemma}
\label{lem:parLL}
If $D,D'\subset\R^3$ are parallel circles,
then
for all $C\in\{\bS(D),\bS(D')\}$
the curve
$C\subset\S^3$ is a circle
and
there exist complex conjugate lines~$L,\oL\subset\cU$
\st
$|C\cap L|=|C\cap\oL|=1$ and $|C\cap L\cap\oL|=0$
.
\end{lemma}

\begin{proof}
Let $S_u\subset\R^3$ denote the sphere with radius~$|u_0|$ and center $(u_1,u_2,u_3)$.
Let $H,H'\subset\R^3$ be the planes and $c,c'\in\R^4$ be the coefficients
\st $D=S_c\cap H$ and $D'=S_{c'}\cap H'$.
Notice that $\pi(\cU)=\set{y\in\P^3}{y_0=y_1^2+y_2^2+y_3^2=0}$ and
\[
\bP(S_u)=\set{y\in\P^3}{-(u_0\,y_0)^2+(y_1-u_1\,y_0)^2+(y_2-u_2\,y_0)^2+(y_3-u_3\,y_0)^2=0}.
\]
Since $\bP(D)=\bP(S_c)\cap\bP(H)$,
it follows from B\'ezout's theorem that~$\bP(D)\cap\pi(\cU)$ consist
of two complex conjugate points.
Therefore, $\pi(\cU)\cap\bP(D)=\{\pi(L),\pi(\oL)\}$
for some complex conjugate lines~$L,\oL\subset\cU$.
Indeed, the complex points in~$\pi(\cU)$ are via the stereographic projection~$\pi$ preimages of complex lines in~$\cU$.
As $H$ and $H'$ are parallel by assumption, we deduce that $\bP(H)\cap\bP(H')$
coincides with the line at infinity spanned by the complex conjugate points~$\pi(L),\pi(\oL)\in\pi(\cU)$.
This implies that
\[
\bP(D)\cap\bP(D')\cap\pi(\cU)=\{\pi(L),\pi(\oL)\}.
\]
Suppose that $B\subset\P^3$ is an irreducible conic
\st $\bR(B)\in\{D,D'\}$.
By definition, $C$ is the Zariski closure of~$\pi^{-1}(B)$.
We first want to show that $\deg C=2$
and thus $|C\cap H|=2$
for a general hyperplane $H\subset \P^4$ (see \citep[\textsection I.7]{1977}).
A straightforward computation shows that
\begin{gather*}
\pi^{-1}(y)=
(
\Delta+y_0^2:
2\,y_0\,y_1:
2\,y_0\,y_2:
2\,y_0\,y_3:
\Delta-y_0^2
)
\quad\text{with}\quad
\Delta:=y_1^2+y_2^2+y_3^2.
\end{gather*}
It follows that the Zariski closure of~$\pi(H\cap\S^3)$
is for some general $\alpha\in\P^4$ equal to the quadratic surface
\[
Q_\alpha:=\set{y\in \P^3}{
\alpha_0\,(\Delta+y_0^2)+
\alpha_1\,y_0\,y_1+
\alpha_2\,y_0\,y_2+
\alpha_3\,y_0\,y_3+
\alpha_4\,(\Delta-y_0^2)=0}.
\]
Since $\pi(\cU)=\set{y\in\P^3}{y_0=\Delta=0}$, we find that
the inverse stereographic projection~$\pi^{-1}$ is not defined at
the conic~$\pi(\cU)$
and $\pi(\cU)\subset Q_\alpha$ for all $\alpha\in \P^4$.
As $|B\cap Q_\alpha|=(\deg B)\cdot(\deg Q_\alpha) =4$
for general~$\alpha\in\P^4$ by B\'ezout's theorem, we deduce that
\[
\deg C=|C\cap H|=|B\cap Q_\alpha|-|B\cap \pi(\cU)|=4-2=2.
\]
The stereographic projection~$\pi$ defines a biregular isomorphism~$C\cong B$ and
thus $C$ is an irreducible conic with real points so that $\bR(C)\subset \R^4$ is a circle.
Hence, $C\subset\S^3$ is a circle by definition.
As $\cU$ is a hyperplane section of $\S^3$,
it follows from B\'ezout's theorem that $|C\cap\cU|=2$.
The complex lines in $\cU$ are contracted by~$\pi$ to complex points in
the conic~$\pi(\cU)$ and thus $|C\cap L|=|C\cap\oL|=1$.
The projection center of~$\pi$ coincides with $L\cap\oL$ and $\deg B=2$,
which implies that $|C\cap L\cap\oL|=0$.
\end{proof}

\begin{proposition}
\label{prp:ABU}
If $\varphi\c\P^1\times\P^1\to X\subset \S^3$ is a biquadratic birational morphism
and there exist circles $A,B\subset\R^3$ \st $X=\bS(A+B)$,
then $\varphi$ is $\cU$-decomposable.
\end{proposition}

\begin{proof}
Let $A_b$ be the circle $A+\{b\}$
and let $C_b:=\bS(A_b)$ for all $b\in B$.
We observe that $A_u$ and $A_v$ are parallel for all $u,v\in B$.
Therefore, it follows from \Cref{lem:parLL} that
there exist complex conjugate lines $L,\oL\subset \cU$
\st for all $b\in B$
the curve~$C_b$ is a circle \st
$|C_b\cap L|=|C_b\cap\oL|=1$ and
$|C_b\cap L\cap\oL|=0$.
Notice that $L\cap\oL=\cU_\R$
and thus if $C_u\cap C_v$ contains a complex point in~$\cU$,
then $|C_u\cap C_v|=2$ as its complex conjugate must also be contained.
We know from Lemmas~\ref{lem:C}\ref{lem:C:a} and \ref{lem:C}\ref{lem:C:b}
that $|C_u\cap C_v|\leq 1$, which shows that $|C_u\cap C_v\cap\cU|=0$
for almost all $u,v\in B$.
In other words,
as we vary~$b\in B$ the circle~$C_b$ reaches
infinitely many complex points in~$L\cup\oL$.
This implies that $L,\oL\subset X\cap\cU$.
For the circles $\{a\}+B$ with $a\in A$,
we obtain analoguously another pair of complex conjugate lines $R,\oR\subset X\cap\cU$.
We deduce that $X\cap\cU$ contains four complex lines
and thus is $\cU$-decomposable by \Cref{lem:LLRR}.
\end{proof}

\begin{proposition}
\label{prp:UE}
If $\varphi\c\P^1\times\P^1\to X\subset \S^3$ is a $\cU$-decomposable biquadratic birational morphism,
then
$\nu\circ\varphi$ is not $\cE$-decomposable for all $\nu\in\aut\S^3$.
\end{proposition}

\begin{proof}
Suppose by contradiction that $\varphi$ is $\cE$-decomposable
for some $\nu\in\aut\S^3$.
It follows from \Cref{lem:LLRR} that
$\nu(\cE)\cap\Sing X$ consists of four complex double lines.
Since $\nu(\cE)_\R=\varnothing$ and each line in the quadratic cone~$\cU$ contains~$\cU_\R$,
we deduce that none of the complex lines in~$\nu(\cE)\cap\Sing X$ are contained in~$\cU$.
The stereographic projection~$\pi$ is an isomorphism outside~$\cU$
and thus $\Sing\pi(X)$ contains at least four complex line components.
Now suppose that $L,F\in \R[x_0,x_1,x_2,x_3]$ are homogeneous polynomials \st
$H=\set{x\in \P^3}{L(x)=0}$ is a general hyperplane
and
$\pi(X)=\set{x\in \P^3}{F(x)=0}$.
Let $C:=\pi(X)\cap H$ be a general hyperplane section.
It follows from Bertini's theorem that $C$ is an irreducible curve
(see \citep[Corollary~III.10.9, Remark~10.9.1]{1977} and \citep[Theorem~2.16]{2007}).
Notice that $\Sing C$ is the set of complex points
where the following Jacobian matrix has lower rank:
\[
\left(
\begin{smallmatrix}
\partial{x_0}F & \partial{x_1}F & \partial{x_2}F & \partial{x_3}F\\
\partial{x_0}L & \partial{x_1}L & \partial{x_2}L & \partial{x_3}L
\end{smallmatrix}
\right).
\]
From this, we deduce that $\Sing C=H\cap \Sing\pi(X)$
and thus $|\Sing C|\geq 4$ by B\'ezout's theorem.
It follows from the genus formula for planar curves (see \citep[Exercises~I.7.2b and IV.1.8a]{1977}) that
\[
\tfrac{1}{2}\cdot(\deg\pi(X)-1)\cdot(\deg\pi(X)-2) - |\Sing C|\geq 0.
\]
Hence, we established that $\deg\pi(X)\geq 5$.
Since $\varphi$ is $\cU$-decomposable, $\cU\cap X$
consists of two left images and two right images.
There are 4 complex intersections
between the corresponding two left fibers and two right fibers in~$\P^1\times\P^1$.
The
left images and right images are complex lines in the quadratic cone~$\cU$
that pass through $p\in\cU_\R$, namely the center of stereographic projection~$\pi$.
Since $p$ has 4 complex points as preimage,
it must be of multiplicity at least~4 and thus $\deg \pi(X)\leq 8-4=4$.
We arrived at a contradiction as we established that $\deg\pi(X)\geq 5$.
This concludes the proof.
\end{proof}

\begin{proposition}
\label{prp:ABCD}
Suppose that $\varphi\c\P^1\times\P^1\to X\subset \S^3$ is a biquadratic birational morphism.
If there exist circles $A,B\subset\R^3$ \st $X=\bS(A+B)$,
then $\nu(X)\neq\bP(C\cdot D)$ for all circles $C,D\subset S^3$ and $\nu\in\aut\S^3$.
\end{proposition}

\begin{proof}
Follows from Propositions~\ref{prp:CDE}, \ref{prp:ABU} and \ref{prp:UE}.
\end{proof}

\begin{remark}
Suppose that $\varphi\c\P^1\times\P^1\to X\subset \S^3$ is a biquadratic birational morphism.
It follows from Propositions~\ref{prp:ECD} and~\ref{prp:CDE}
that $\varphi$ is $\cE$-decomposable if and only if
there exist circles $C,D\subset S^3$ \st $X=\bP(C\cdot D)$.
It follows from Propositions~\ref{prp:UAB} and~\ref{prp:ABU}
that $\varphi$ is $\cU$-decomposable if and only if
there exist circles $A,B\subset\R^3$ \st $X=\bS(A+B)$.
\END
\end{remark}

\section*{Acknowledgements}
\addcontentsline{toc}{section}{Acknowledgements}

We thank Mikhail Skopenkov and the anonymous referees for their insightful comments,
numerous corrections, and helpful suggestions for improvement.
This work was supported by the Austrian Science Fund (FWF)
projects~P33003 and~P36689.

\bibliography{double-line}

\begin{thebibliography}{10}

\bibitem{2014rin}
B.~Bastl, B.~J{\"u}ttler, M.~L{\'a}vi{\v{c}}ka, T.~Schulz, and
  Z.~{\v{S}}{\'{\i}}r.
\newblock On the parameterization of rational ringed surfaces and rational
  canal surfaces.
\newblock {\em Math. Comput. Sci.}, 8(2):299--319, 2014.
\newblock \href {https://doi.org/10.1007/s11786-014-0192-y}
  {\path{doi:10.1007/s11786-014-0192-y}}.

\bibitem{2010ber}
M.~Berger.
\newblock {\em Geometry revealed. {A} {Jacob}'s ladder to modern higher
  geometry}.
\newblock Springer, 2010.
\newblock \href {https://doi.org/10.1007/978-3-540-70997-8}
  {\path{doi:10.1007/978-3-540-70997-8}}.

\bibitem{1980}
R.~Blum.
\newblock Circles on surfaces in the {Euclidean} space.
\newblock Geometry and differential geometry, {Proc}. {Conf}., {Haifa} 1979,
  {Lect}. {Notes} {Math}. 792, 213-221, 1980.

\bibitem{2011arc}
P.~Bo, H.~Pottmann, M.~Kilian, W.~Wang, and J.~Wallner.
\newblock Circular arc structures.
\newblock {\em ACM Trans. Graph.}, 30:101, 07 2011.
\newblock \href {https://doi.org/10.1145/2010324.1964996}
  {\path{doi:10.1145/2010324.1964996}}.

\bibitem{2012ort}
A.I. Bobenko and E.~Huhnen-Venedey.
\newblock Curvature line parametrized surfaces and orthogonal coordinate
  systems: discretization with {Dupin} cyclides.
\newblock {\em Geom. Dedicata}, 159:207--237, 2012.
\newblock \href {https://doi.org/10.1007/s10711-011-9653-5}
  {\path{doi:10.1007/s10711-011-9653-5}}.

\bibitem{2016new}
C.C.-A. Cheng and T.~Sakkalis.
\newblock On new types of rational rotation-minimizing frame space curves.
\newblock {\em J. Symb. Comput.}, 74:400--407, 2016.
\newblock \href {https://doi.org/10.1016/j.jsc.2015.08.005}
  {\path{doi:10.1016/j.jsc.2015.08.005}}.

\bibitem{1998}
H.S.M. Coxeter.
\newblock {\em Non-{E}uclidean geometry}.
\newblock Spectrum. MAA, sixth edition, 1998.

\bibitem{2012dol}
I.V. Dolgachev.
\newblock {\em Classical algebraic geometry. {A} modern view}.
\newblock Cambridge University Press, 2012.
\newblock \href {https://doi.org/10.1017/CBO9781139084437}
  {\path{doi:10.1017/CBO9781139084437}}.

\bibitem{2022abc}
T.~Duarte~Guerreiro, Z.~Li, and J.~Schicho.
\newblock Classification of higher mobility closed-loop linkages.
\newblock {\em Ann. Mat. Pura Appl.}, 2022.
\newblock \href {https://doi.org/10.1007/s10231-022-01258-y}
  {\path{doi:10.1007/s10231-022-01258-y}}.

\bibitem{2008}
G.~Gentili and C.~Stoppato.
\newblock Zeros of regular functions and polynomials of a quaternionic
  variable.
\newblock {\em Mich. Math. J.}, 56(3):655--667, 2008.
\newblock \href {https://doi.org/10.1307/mmj/1231770366}
  {\path{doi:10.1307/mmj/1231770366}}.

\bibitem{2021slice}
G.~Gentili, C.~Stoppato, and T.~Trinci.
\newblock Zeros of slice functions and polynomials over dual quaternions.
\newblock {\em Trans. Am. Math. Soc.}, 374(8):5509--5544, 2021.
\newblock \href {https://doi.org/10.1090/tran/8346}
  {\path{doi:10.1090/tran/8346}}.

\bibitem{1966}
B.~Gordon and T.S. Motzkin.
\newblock On the zeros of polynomials over division rings.
\newblock {\em Trans. Am. Math. Soc.}, 116:218--226, 1966.
\newblock \href {https://doi.org/10.2307/1994114} {\path{doi:10.2307/1994114}}.

\bibitem{1977}
R.~Hartshorne.
\newblock {\em Algebraic geometry}.
\newblock Springer, 1977.
\newblock Graduate Texts in Mathematics, No. 52.

\bibitem{1995}
T.~Ivey.
\newblock Surfaces with orthogonal families of circles.
\newblock {\em Proc. Am. Math. Soc.}, 123(3):865--872, 1995.
\newblock \href {https://doi.org/10.2307/2160812} {\path{doi:10.2307/2160812}}.

\bibitem{1968}
F.~Klein.
\newblock Vorlesungen {\"u}ber nicht-{E}uklidische {Geometrie}.
\newblock Springer, 1968.

\bibitem{2007}
J.~Koll{\'a}r.
\newblock {\em Lectures on resolution of singularities}, volume 166 of {\em
  Ann. Math. Stud.}
\newblock Princeton University Press, 2007.
\newblock \href {https://doi.org/10.1515/9781400827800}
  {\path{doi:10.1515/9781400827800}}.

\bibitem{2018kol}
J.~Koll{\'a}r.
\newblock Quadratic solutions of quadratic forms.
\newblock In {\em Local and global methods in algebraic geometry}, pages
  211--249. American Mathematical Society, 2018.
\newblock \href {https://doi.org/10.1090/conm/712/14348}
  {\path{doi:10.1090/conm/712/14348}}.

\bibitem{2014bez}
R.~Krasauskas and S.~Zub{\.e}.
\newblock Rational {B{\'e}zier} formulas with quaternion and {Clifford} algebra
  weights.
\newblock In {\em SAGA}, pages 147--166. Springer, 2014.
\newblock \href {https://doi.org/10.1007/978-3-319-08635-4_8}
  {\path{doi:10.1007/978-3-319-08635-4_8}}.

\bibitem{2020kin}
R.~Krasauskas and S.~Zub{\.e}.
\newblock Kinematic interpretation of {Darboux} cyclides.
\newblock {\em Comput. Aided Geom. Des.}, 83:16, 2020.
\newblock \href {https://doi.org/10.1016/j.cagd.2020.101945}
  {\path{doi:10.1016/j.cagd.2020.101945}}.

\bibitem{2022n1}
J.~Lercher, D.F. Scharler, H.-P. Schröcker, and J.~Siegele.
\newblock Factorization of quaternionic polynomials of bi-degree $(n,1)$.
\newblock {\em Beitr. Algebra Geom.}, 2022.
\newblock \href {https://doi.org/10.1007/s13366-022-00629-z}
  {\path{doi:10.1007/s13366-022-00629-z}}.

\bibitem{2022tech}
J.~Lercher and H.-P. Schr{\"o}cker.
\newblock A multiplication technique for the factorization of bivariate
  quaternionic polynomials.
\newblock {\em Adv. Appl. Clifford Algebr.}, 32(1):23, 2022.
\newblock \href {https://doi.org/10.1007/s00006-021-01194-9}
  {\path{doi:10.1007/s00006-021-01194-9}}.

\bibitem{2018kempe}
Z.~Li, J.~Schicho, and H.-P. Schr{\"o}cker.
\newblock Kempe's universality theorem for rational space curves.
\newblock {\em Found. Comput. Math.}, 18(2):509--536, 2018.
\newblock \href {https://doi.org/10.1007/s10208-017-9348-x}
  {\path{doi:10.1007/s10208-017-9348-x}}.

\bibitem{2019}
Z.~Li, J.~Schicho, and H.-P. Schr{\"o}cker.
\newblock Factorization of motion polynomials.
\newblock {\em J. Symb. Comput.}, 92:190--202, 2019.
\newblock \href {https://doi.org/10.1016/j.jsc.2018.02.005}
  {\path{doi:10.1016/j.jsc.2018.02.005}}.

\bibitem{2024}
Z.~Li, H.-P. Schr\"ocker, M.~Skopenkov, and D.F. Scharler.
\newblock {Motion Polynomials Admitting a Factorization with Linear Factors}.
\newblock 2024.
\newblock \href{https://arxiv.org/abs/2209.02306}{arXiv:2209.02306}.

\bibitem{2021circ}
N.~Lubbes.
\newblock Surfaces that are covered by two pencils of circles.
\newblock {\em Math. Z.}, 299(3-4):1445--1472, 2021.
\newblock \href {https://doi.org/10.1007/s00209-021-02713-x}
  {\path{doi:10.1007/s00209-021-02713-x}}.

\bibitem{2023trans}
N.~Lubbes.
\newblock Translational and great {D}arboux cyclides.
\newblock {\em C. R. Math. Acad. Sci. Paris}, 362:413--448, 2024.
\newblock \href {https://doi.org/10.5802/crmath.603}
  {\path{doi:10.5802/crmath.603}}.

\bibitem{2018kin}
N.~Lubbes and J.~Schicho.
\newblock Kinematic generation of {Darboux} cyclides.
\newblock {\em Comput. Aided Geom. Des.}, 64:11--14, 2018.
\newblock \href {https://doi.org/10.1016/j.cagd.2018.06.001}
  {\path{doi:10.1016/j.cagd.2018.06.001}}.

\bibitem{2019boh}
P.C. López-Custodio and J.S. Dai.
\newblock {Design of a Variable-Mobility Linkage Using the Bohemian Dome}.
\newblock {\em Journal of Mechanical Design}, 141(9), 04 2019.
\newblock \href {https://doi.org/10.1115/1.4042845}
  {\path{doi:10.1115/1.4042845}}.

\bibitem{2012web}
H.~{Pottmann}, L.~{Shi}, and M.~{Skopenkov}.
\newblock {Darboux cyclides and webs from circles}.
\newblock {\em {Comput. Aided Geom. Des.}}, 29(1):77--97, 2012.
\newblock \href {https://doi.org/10.1016/j.cagd.2011.10.002}
  {\path{doi:10.1016/j.cagd.2011.10.002}}.

\bibitem{2014rt}
D.~Roegel.
\newblock {The ''Villarceau circles'' in Uhlberger's staircase (ca. 1580)}.
\newblock Research report, 2014.
\newblock URL: \url{https://inria.hal.science/hal-00941465}.

\bibitem{2001}
J.~Schicho.
\newblock The multiple conical surfaces.
\newblock {\em Beitr. Algebra Geom.}, 42(1):71--87, 2001.

\bibitem{2020}
J.~Siegele, D.F. Scharler, and H.-P. Schr{\"o}cker.
\newblock Rational motions with generic trajectories of low degree.
\newblock {\em Comput. Aided Geom. Des.}, 76:10, 2020.
\newblock \href {https://doi.org/10.1016/j.cagd.2019.101793}
  {\path{doi:10.1016/j.cagd.2019.101793}}.

\bibitem{2018sko}
M.~Skopenkov and R.~Krasauskas.
\newblock Surfaces containing two circles through each point.
\newblock {\em Math. Ann.}, 373(3-4):1299--1327, 2019.
\newblock \href {https://doi.org/10.1007/s00208-018-1739-z}
  {\path{doi:10.1007/s00208-018-1739-z}}.

\bibitem{1987}
N.~Takeuchi.
\newblock A closed surface of genus one in {{\(E^ 3\)}} cannot contain seven
  circles through each point.
\newblock {\em Proc. Am. Math. Soc.}, 100:145--147, 1987.
\newblock \href {https://doi.org/10.2307/2046136} {\path{doi:10.2307/2046136}}.

\bibitem{2000}
N.~Takeuchi.
\newblock Cyclides.
\newblock {\em Hokkaido Math. J.}, 29(1):119--148, 2000.
\newblock \href {https://doi.org/10.14492/hokmj/1350912960}
  {\path{doi:10.14492/hokmj/1350912960}}.

\bibitem{2019mor}
M.~Zhao, X.~Jia, C.~Tu, B.~Mourrain, and W.~Wang.
\newblock Enumerating the morphologies of non-degenerate {Darboux} cyclides.
\newblock {\em Comput. Aided Geom. Des.}, 75:15, 2019.
\newblock \href {https://doi.org/10.1016/j.cagd.2019.101776}
  {\path{doi:10.1016/j.cagd.2019.101776}}.

\end{thebibliography}

Johanna Frischauf
\\Department of Basic Sciences in Engineering Sciences, University of Innsbruck,
\\Technikerstr.~13, 6020 Innsbruck, Austria
\\\texttt{johanna.frischauf@uibk.ac.at}
\\[2mm]Niels Lubbes
\\Institute for Algebra, Johannes Kepler University, Linz, Austria
\\\texttt{info@nielslubbes.com}
\\[2mm]Hans-Peter Schröcker
\\Department of Basic Sciences in Engineering Sciences, University of Innsbruck,
\\Technikerstr.~13, 6020 Innsbruck, Austria
\\\texttt{hans-peter.schroecker@uibk.ac.at}

\end{document}